\documentclass[a4paper, 11pt]{article}
\usepackage{authblk}
\usepackage{latexsym,amsmath,amsfonts,amssymb,amsbsy,amsopn,amstext,amsthm}
\usepackage{epsfig,epstopdf,graphicx}
\usepackage[mathscr]{eucal}
\usepackage{color,multicol,makeidx}
\usepackage{bm,mathrsfs}
\usepackage{float}
\usepackage{mathrsfs}
\usepackage{tabu}
\usepackage{enumerate}

\def\dref#1{(\ref{#1})}
\def\eq{\displaystyle\stackrel\triangle=}
\def\xra{\xrightarrow}

\newcommand{\Var}{\ensuremath{\mathrm{Var}}}

\newcommand{\diag}{\ensuremath{\mathrm{diag}}}
\DeclareMathOperator*{\argmin}{arg\,min}

\newtheorem{assum}{Assumption}
\newtheorem{thm}{Theorem}
\newtheorem{prop}{Proposition}
\newtheorem{lem}{Lemma}
\newtheorem{rem}{Remark}

\newtheorem{defn}{Definition}

\providecommand{\keywords}[1]
{
	\small	
	\textbf{\textit{Keywords---}} #1
}

	\title{On the Asymptotic Optimality of Cross-Validation based Hyper-parameter Estimators for Regularized Least Squares Regression Problems}

\vspace{5mm}

\author[1]{Biqiang Mu\\
	Key Laboratory of Systems and Control \\
	Institute of Systems Science\\
	Academy of Mathematics and Systems Science \\
	Chinese Academy of Sciences\\
	Beijing 100190, China \\
	(bqmu@amss.ac.cn)\\
	
	\vspace{5mm}
	Tianshi Chen\\
	School of Data Science and
	Shenzhen Research Institute of Big Data\\
	The Chinese University of Hong Kong\\
	Shenzhen 518172,
	China\\
	(tschen@cuhk.edu.cn)\\
	
	\vspace{5mm}
	Lennart Ljung\\
	Division of Automatic Control\\
	Department of Electrical Engineering\\
	Link\"oping University\\
	Link\"oping SE-58183,
	Sweden\\
	(ljung@isy.liu.se)
}
\begin{document}
	
	\maketitle

	\newpage
	\begin{abstract}
		The asymptotic optimality (a.o.) of various hyper-parameter estimators with different optimality criteria has been studied in the literature   
		for regularized least squares regression problems. The estimators include e.g., the maximum (marginal) likelihood method, $C_p$ statistics, and generalized cross validation method, and the optimality criteria are based on e.g., the inefficiency, the expectation inefficiency and the risk. In this paper, we consider the 
		regularized least squares regression problems with fixed number of regression parameters, choose the optimality criterion based on the risk, and study the a.o. of several cross validation (CV) based hyper-parameter estimators including 
		the leave $k$-out CV method, generalized CV method, $r$-fold CV method and hold out CV method. We find the former three methods can be a.o. under mild assumptions, but not the last one, and we use Monte Carlo simulations to illustrate the efficacy of our findings.
	\end{abstract}


\section{Introduction}

Consider the linear regression model
\begin{align}\label{lrm0}
y_i  = x_i^T \beta + \varepsilon_i, \quad i=1,\dots, n,
\end{align}
where $n\in\mathbb N$ is the number of data, $y_i\in\mathbb R,x_i\in{\mathbb R}^p,\varepsilon_i\in\mathbb R$, are the $i$th observation, known covariate and noise, respectively, $\beta\in\mathbb R^p$ is a $p$-dimensional vector of unknown regression parameters, and the covariates $x_i$, $i=1,\dots, n$, are assumed to be deterministic and the noises $\varepsilon_i$, $i=1,\dots, n$, are i.i.d. with mean 0 and variance $0<\sigma^2<\infty$. 

For convenience, \eqref{lrm0} can be written in a matrix-vector format as follows
\begin{align}
y = X  \beta + \varepsilon, \label{lrm}
\end{align}
where $y=[y_1,\cdots,y_n]^T, X=[x_1,\cdots,x_n]^T$, and $\varepsilon=[\varepsilon_1,\cdots,\varepsilon_n]^T$. Provided that $X$ has full column rank, the least squares (LS) estimator of $\beta$ is
\begin{align}
\widehat{\beta}_{\rm LS}=\argmin_{\beta \in \mathbb{R}^p}\|y-X\beta\|^2
= (X^TX)^{-1}X^T y, \label{ls}
\end{align} where $\|\cdot\|$ is the Euclidean norm. It is well-known that 
the LS estimator $\widehat{\beta}_{\rm LS}$ is unbiased and has covariance matrix  $\sigma^2 (X^TX)^{-1}$, and moreover, $\widehat{\beta}_{\rm LS}$ has a large variance if $X^TX$ is ill-conditioned, e.g., \cite{Rao2018}. 
One way to mitigate the large variance is to add a quadratic regularization term in the LS criterion \eqref{ls}, leading to the regularized least squares (RLS) estimator
\begin{subequations}
	\label{rls0}
	\begin{align}
	\widehat{\beta}(\eta)
	&=\argmin_{\beta \in \mathbb{R}^p}\|y-X\beta\|^2
	+\gamma\beta ^T K(\eta)^{-1}\beta\\
	\label{rls20}
	&=(X^T   X   +  \gamma K(\eta)^{-1})	^{-1}X^Ty,
	\end{align}
\end{subequations}
where $\gamma>0$ is the regularization parameter, $K(\eta)\in\mathbb R^{p\times p}$ is a  suitably designed positive definite matrix
and $\eta\in\Omega$ is the parameter used to parametrize $K(\eta)$ with $\Omega\subset \mathbb R^m$ being the set in which $\eta$ takes values.
The simplest case of
\eqref{rls0} is to let $\gamma=1$ and $K(\eta)= I_p/\eta$  with $\eta\geq0$, i.e.,
\begin{align}
\widehat{\beta}_{\rm R}(\eta)= (X^TX + \eta I_p)^{-1}X^T y, \label{rr}
\end{align}
which is the well-known
ridge regression estimator that can have smaller mean squared error (MSE) than the LS estimator \eqref{ls}  if $\eta$ is chosen properly, e.g., \cite{Hoerl1970}.

To simply the notations in the following discussions,  we let $\gamma=\sigma^2$ in \eqref{rls0} and consider the following RLS estimator in the sequel
\begin{subequations}
	\label{rls}
	\begin{align}
	\widehat{\beta}(\eta)
	&=\argmin_{\beta \in \mathbb{R}^p}\|y-X\beta\|^2
	+\sigma^2\beta ^T K(\eta)^{-1}\beta\\
	\label{rls2}
	&=(X^T   X   +  \sigma^2 K(\eta)^{-1})	^{-1}X^Ty.
	\end{align}
\end{subequations}
\begin{rem}
	The matrix $\gamma K(\eta)^{-1}$ in \eqref{rls0} is called the regularization matrix, and moreover, for any given $\gamma$, there exists a $\bar\eta$ such that $\gamma K(\eta)^{-1}=K(\bar\eta)^{-1}$ by merging $\gamma$ and the scale factor of $K(\eta)$, which implies that the value of $\gamma$
	does not matter for the tuning of the RLS estimator \eqref{rls0}. However, letting $\gamma=\sigma^2$ can lighten the notations in the asymptotic analysis,
	where $\sigma^2$ is assumed to be known in the following discussions. 
	
\end{rem}
The parameter $\eta$ is called the hyper-parameter and its estimation can be handed by many methods such as the maximum (marginal) likelihood (ML) method \cite{Maritz1989}, $C_p$ statistics \cite{Mallows1973}, Stein's unbiased risk
estimator (SURE) \cite{Stein1981}, and cross validation (CV) methods \cite{Hastie2009}, leading to many estimators of $\eta$. The asymptotic optimality (a.o.) of these estimators with different optimality criteria  has been studied in the literature, e.g., \cite{Craven1979,Golub1979,Wahba1985,Speckman1985,Li1986,Girard1991}.

\subsection{Optimality Criteria}\label{sec:oc}
In what follows, we briefly review some optimality criteria including the inefficiency, the expectation inefficiency and the risk.
These optimality criteria are all based on the \emph{true loss} (also called the predictive mean squared error), e.g. \cite{Girard1991,Li1986,Wahba1985},  for the estimation of $X\beta$ by $X\widehat{\beta}(\eta)$, which is defined as follows
\begin{align}
L_n(\eta) = \frac{1}{n} \|X\beta - X \widehat{\beta}(\eta)\|^2\label{tl}.
\end{align}

\subsubsection{The inefficiency}
For the ridge regression estimator \dref{rr}, the inefficiency of an estimator $\widehat{\eta}_n$ 
of $\eta$   is  defined  as
$L_n(\widehat{\eta}_n)/\inf\limits_{\eta\geq 0} L_n(\eta)$
and the estimator $\widehat{\eta}_n$ is said to be asymptotically optimal (a.o.) in \cite{Li1986,Girard1991} if
\begin{align}
\frac{L_n(\widehat{\eta}_n)}{\inf\limits_{\eta\geq 0} L_n(\eta)}\xra{}1 \label{aor}
\end{align}
in probability as $n\xra{}\infty$. It was shown in \cite{Li1986,Girard1991} that for the ridge regression estimator \dref{rr}, when $p$ is a function of $n$ in \eqref{lrm0} and $p\rightarrow\infty$ as $n\rightarrow\infty$,
the $C_p$ and GCV based hyper-parameter estimators of $\eta$ and their fast randomized versions are all a.o. under some assumptions.

\subsubsection{The expectation inefficiency}
For the ridge regression estimator \dref{rr}, the expectation inefficiency of  an estimator $\widehat{\eta}_n$ 
of $\eta$ is  defined as
\begin{align}
\mathcal{I}_n=\frac{E (L_n(\widehat{\eta}_n))}{\inf\limits_{\eta\geq 0} E(L_n(\eta))}
\label{ei}
\end{align} where $E(\cdot)$ denotes the mathematical expectation with respective to the noise distribution, and the estimator $\widehat{\eta}_n$ is said to be a.o. in \cite{Golub1979} if $\mathcal{I}_n\xra{}1$ as $n\xra{}\infty$.
It was shown in \cite{Golub1979} that for the ridge regression \dref{rr} with  either fixed $p$ or $p>n$, the ``expectation'' GCV estimator is a.o..
The definitions of the expectation inefficiency and its corresponding a.o. are also adopted for the tuning of regularization parameter in spline smoothing \cite{Craven1979,Speckman1985,Wahba1985}.
It was first shown in \cite{Craven1979} that the ``expectation'' GCV based hyper-parameter estimator of $\eta$ is a.o. and then shown in \cite{Speckman1985} that the GCV hyper-parameter estimator of $\eta$ itself is a.o..  Besides, the generalized ML and GCV methods are compared in \cite{Wahba1985} in terms of the expectation efficiency under different assumptions on the smoothness of the underlying function to be estimated.

\subsubsection{The risk}\label{sec:risk}

For the RLS estimator $\widehat{\beta}(\eta)$ in \eqref{rls}, the risk is defined in e.g., \cite{Craven1979,Golub1979,Wahba1985} as
$R_n(\eta) = E (L_n(\eta)),$
and the corresponding optimal hyper-parameter estimator $\eta^*_n$ is defined as
\begin{align}\label{ophy_risk}
\eta^*_n=\argmin_{\eta\in \Omega} R_n(\eta)
\end{align}
where ``$\arg\min$'' denotes the set of global minima of $R_n(\eta)$, i.e., 
\begin{align}\label{eq:D}
\arg\min_{\eta \in \Omega} R_n(\eta) = \big\{\eta|\eta \in
\Omega, R_n(\eta)=\min_{\eta'\in \Omega }R_n(\eta')\big\}.
\end{align}
Then an estimator $\widehat{\eta}_n$ of $\eta$ is said to be a.o. in \cite{Mu2018a,Mu2018c1} if $\eta^*_n$ has a limit and $\widehat{\eta}_n$ converges in probability to the limit as $n\xra{}\infty$. It was shown in \cite{Mu2018a,Mu2018c1} that for fixed $p$, the SURE estimator is a.o. under some assumptions, but the ML estimator is in general not a.o.. The optimal hyperparameter estimator $\eta^*_n$ 
was also introduced in \cite{Speckman1985} for spline smoothing in nonparametric regression models to study the a.o. in the sense of the expectation inefficiency of the GCV estimator.

\subsection{Problem Statement}
In this paper, we consider the linear regression model \eqref{lrm0} or equivalently \eqref{lrm}  with fixed $p$ and our goal is to study the asymptotic optimality (a.o.) of various CV based hyper-parameter estimators of $\eta$, where the regression parameter vector $\beta$ is estimated using the RLS estimator  $\widehat{\beta}(\eta)$. The class of CV methods is perhaps the simplest and most widely used class of methods for parameter estimation \cite{Hastie2009}.  
The CV methods that will be considered here and reviewed shortly include
the leave $k$-out CV (LKOCV) method, generalized CV (GCV) method, $r$-fold CV (RFCV) method and hold out CV (HOCV) method. 

As can be seen from the brief review in Section \ref{sec:oc}, the case where $p$ is a function of $n$ and $p\rightarrow\infty$ as $n\rightarrow\infty$, has been well studied for the a.o. of the GCV based hyper-parameter estimator, e.g., \cite{Li1986,Girard1991,Golub1979}, but not the case with a fixed $p$. 

It is worth to stress that the case with a fixed $p$ is fundamental in linear regression problems, e.g., \cite{DW1980,DS1981,Rao2018}, and has wide applications in various disciplines of science and engineering, e.g., in the area of systems and control and machine learning, where $p$ corresponds to the number of basis functions, e.g., \cite{Hofmann2008,Ljung1999}. Moreover, when $X^TX$ is severely ill-conditioned, the quadratic regularization term $\sigma^2\beta ^T K(\eta)^{-1}\beta$ in \eqref{rls} plays a key role in ensuring a high quality RLS estimator $\widehat{\beta}(\eta)$, which depends on both the design of a suitable regularization matrix $\sigma^2K(\eta)^{-1}$ and the corresponding hyper-parameter estimator $\hat\eta_n$, e.g., \cite{PDCDL14} for illustrations in the area of systems and control. 
Finally, it should be noted that the analysis methods in \cite{Li1986} and its subsequent papers, e.g., \cite{Dobriban2018,Girard1991,Shao1997,Xu2012} consider the case, where $p$ is a function of $n$ in \eqref{lrm0} and $p\rightarrow\infty$ as $n\rightarrow\infty$,
and moreover, require the assumption $\inf_{\eta}nR_n(\eta)\xra{}\infty$ as $n\xra{}\infty$, that is essential for deriving the a.o. of the $C_p$ and GCV estimators. However, for fixed $p$, $nR_n(\eta)$ is bounded from above by a constant and thus the analysis methods in \cite{Li1986} and its subsequent papers can not be applied here to deal with the case with a fixed $p$.

The rest of the paper is organized as follows.
We consider the selection of optimality criteria in Section \ref{sec:optimality}. Then we briefly review the CV methods and analyze their a.o. in Sections
\ref{sec:cv} and \ref{sec:main}, respectively. 
Finally, we conclude the paper in Section \ref{sec7}.
All proofs of the theoretical results and numerical simulation  are
postponed to the Appendices.

\section{Selection and Definition of Optimality Criterion}\label{sec:optimality}

To study the a.o. of the aforementioned CV methods, we need to first choose an optimality criterion. It turns out that for the linear regression model \eqref{lrm} with a fixed $p$, 
if the inefficiency and the expected inefficiency are chosen, some trivial estimators can be shown to be a.o. accordingly, as summarized in the following proposition.
\begin{prop}
	\label{thm21}
	Consider the ridge regression estimator \dref{rr}.
	Suppose that
	\begin{enumerate}[1)]
		\item The dimension $p$ of the parameter $\beta$ is fixed and $X^TX/n=O(1)$,
		\item An estimator $\widehat{\eta}_n$ of $\eta$ is bounded in probability.
	\end{enumerate}
	Then the estimator $\widehat{\eta}_n$ is such that 
	$L_n(\widehat{\eta}_n)/\inf\limits_{\eta\geq 0} L_n(\eta)\xra{}1$
	in probability as $n\xra{}\infty$ and $\mathcal{I}_n=E L_n(\widehat{\eta}_n)/\inf\limits_{\eta\geq 0} EL_n(\eta) \xra{}1$ as $n\xra{}\infty$.
\end{prop}

\begin{rem}
	Propositions \ref{thm21} still holds if the assumption $X^TX/n=O(1)$ is relaxed to $X^TX/n^\delta=O(1)$ for some $\delta>0$. The reason why Proposition \ref{thm21} holds is that both the true  loss and the risk can be decomposed as summations of a large constant term independent of $\eta$ and some smaller terms dependent on $\eta$.
\end{rem}

Therefore, we choose to use the optimality criterion based on the risk in Section \ref{sec:risk} to study the a.o. of the CV methods. To this goal, we need to first guarantee the existence of the limit of $\eta^*_n$, which has been discussed in \cite{Mu2018a,Mu2018c1} and summarized in the following Lemma \ref{lm1}. To state the Lemma, we need the following assumptions.



\begin{assum}
	\label{assum1}
	\begin{enumerate}[1)]
		
		\item The dimension $p$ of the parameter $\beta$ is fixed.
		\item $X^TX/n$ converges to a  positive definite matrix $\Sigma$ as $n\xra{}\infty$.		
		\item $K(\eta)$ is a  symmetric and  positive definite matrix  parameterized by $\eta$ and $\|K(\eta)^{-1}\|< \infty$ for any interior point $\eta\in \Omega$.
		
	\end{enumerate}
\end{assum}

\begin{rem}
	Assumption 1.1 was adopted for linear regression models in  \cite{Golub1979,Erdal1983,Zhang1993,Shao1993}.
	Assumption 1.2 is stronger than $X^TX/n=O(1)$ in Propositions \ref{thm21}, it guarantees that an affine transform of the risk $R_n(\eta)$ has an explicit limiting expression depending on $\Sigma$, which is key to derive the limit of $\eta_n^*$.
	Assumption  1.3 can cover the ridge regression estimator \cite{Hoerl1970}, generalized ridge regression estimator \cite{Hoerl1970}, and the ridge-regression estimator \cite{Hall1987}.

\end{rem}

Under Assumption \ref{assum1}, there holds that 
\begin{align*}
W_n(\eta)& \eq n^{2} (  R_n(\eta) - \sigma^2p/n  )/\sigma^4\\
&\xra{}\beta^TK(\eta)^{-1}\Sigma^{-1}K(\eta)^{-1}\beta\!
-2{\rm Tr}
(\Sigma^{-1}\!K(\eta)^{-1})\eq W(\eta)
\end{align*}
as $n
\xra{}\infty$.
Note that $W_n(\eta)$ is an affine transform independent of the $\eta$ and thus, 
$
\eta_n^*=\argmin_{\eta \in \Omega} R_n(\eta)
=\argmin_{\eta \in \Omega}
W_n(\eta).$

We also need the following assumption. 
\begin{assum}\label{ass2}
	The global minima of  $W(\eta)$, i.e., 
	$\eta^*\eq \argmin_{\eta \in \Omega} W(\eta)$
	exist and moreover, are isolated and interior points of 
	$\Omega $.
\end{assum}

Then we are able to state the following lemma derived in \cite{Mu2018a}, which guarantees the existence of the limit of $\eta_n^*$.
\begin{lem}\label{lm1}
	Suppose that Assumptions \ref{assum1} and \ref{ass2} hold.
	Then, it holds that
	$\eta_n^*
	\xra{}
	\eta^* 
	$
	in probability as $n\xra{}\infty$,
	where the convergence 
	should be understood as  the set convergence in the following sense
	$\inf_{a\in \eta_n^*,b\in\eta^*} \|a - b\|\xra{}0$.
\end{lem}

Now we let $\mathscr{C}_n(\eta)$ be a hyper-parameter estimation criterion and 
$\widehat{\eta}_n$ be the corresponding hyper-parameter estimator that minimizes $\mathscr{C}_n(\eta)$, i.e., 
$\widehat\eta_n=\argmin_{\eta\in \Omega} \mathscr{C}_n(\eta),$ 
where ``$\arg\min$'' has been defined in \eqref{eq:D}, and we define two kinds of a.o. of the hyper-parameter estimator $\widehat{\eta}_n$.

\begin{defn}
	\label{defn1}
	The estimator $\widehat{\eta}_n$  is said to be asymptotically optimal (a.o.) if  
	$\widehat{\eta}_n\xra{}\eta^*$ in probability as $n\xra{}\infty$,
	i.e.,
	$\inf_{a\in \widehat{\eta}_n,b\in\eta^*} \|a - b\|\xra{}0$ in probability
	as $n\xra{}\infty$.
\end{defn}

\begin{defn}
	\label{defn2}
	Denote the sets consisting of the roots of
	\begin{align}
	{\rm Tr}\left(
	\frac{\partial \mathscr{C}_n(\eta)}{\partial K}\frac{\partial K(\eta)}{\partial \eta_i}\right)=0, ~i=1,\cdots,m,
	\label{rootset1}
	\end{align}
	and
	\begin{align}
	&{\rm Tr}\left(
	\frac{\partial W(\eta)}{\partial K}\frac{\partial K(\eta)}{\partial \eta_i}\right)=0, ~~i=1,\cdots,m.
	\label{rootset2}
	\end{align}
	by $\widehat{\psi}_n$ and $\psi^*$,
	respectively. 
	Then, the estimator $\widehat{\eta}_n$ is said to be asymptotically optimal (a.o.) to the first order, if  
	$\widehat{\psi}_n\xra{}\psi^*$ in probability as $n\xra{}\infty$, i.e.,
	$\inf_{a\in \widehat{\psi}_n,b\in\psi^*} \|a - b\|\xra{}0$  in probability
	as $n\xra{}\infty$.
\end{defn} 

\begin{rem}
	Definition \ref{defn2} considers the limiting behavior of the roots of the first order optimality condition of $\mathscr{C}_n(\eta)$. Such kind of limiting behavior was studied before for the roots of the maximum likelihood estimators, e.g., \cite{Perlman1983}, \cite[Theorem 3.7 on p. 447]{Lehmann1998}.
\end{rem}

The following proposition shows that Definition  \ref{defn1} is equivalent to Definition  \ref{defn2} for the ridge regression estimator that is mainly studied in the existing literature, e.g., \cite{Craven1979,Golub1979,Wahba1985,Speckman1985,Li1986,Girard1991}.

\begin{prop}	\label{prop3}
	For the ridge regression estimator \dref{rr},  
	the sets $\eta^*$ and $\psi^*$ are identical and both of them contains
	a single point 
	$\eta^* = \psi^*
	=\frac{\sigma^2{\rm Tr}\big(\Sigma^{-1} \big)}{\beta^T\Sigma^{-1}\beta}.$
	Therefore, Definition \ref{defn1} and Definition \ref{defn2} are equivalent.
	In particular, for the orthogonal regressors $x_i,1\leq i \leq n$, namely, $X^TX=nI_p$ ($\Sigma=I_p$), there holds $\eta^* = \psi^*=\sigma^2p/\beta^T\beta$.
\end{prop}

\begin{rem}
	It is worth pointing out that a hyper-parameter estimator $\widehat{\eta}_n$ that is a.o. in the sense of Definition \ref{defn1} is also a.o. in the sense of expectation inefficiency \dref{ei}, but the converse is not true (See Proposition \ref{thm21}).
\end{rem}

\begin{rem} One may wonder why we did not choose the true loss \eqref{tl}. This is because  for the linear regression model \eqref{lrm} with fixed $p$, the true loss is a random variable as $n\rightarrow\infty$ and it can be proved that the limit of the true loss \eqref{tl} does not exist. It is also interesting to note that 
	the true loss and the risk are asymptotically equivalent for the linear regression model with $p\xra{}\infty$ as $n\xra{}\infty$, c.f., the equation (2.3) of \cite{Li1986}.
\end{rem}

\section{Brief Review of CV Methods}
\label{sec:cv}
In this section, we briefly review the LKOCV, GCV, RFCV and HOCV. Before proceeding to the details, we first introduce some conventions and notations. For brevity, we sometimes omit the parameter $\eta$ in $K(\eta)$ and $\eta$ in $\widehat{\beta}(\eta)$, and etc., and also the subscript $n$ for the estimators of $\eta$ to be introduced below.
We let the data $\{y_i,x_i^T,$ $i=1,\cdots,n\}$ be split into two disjoint parts: the training data $\{y_i,x_i^T,i\in s^c\}$ and validation data $\{y_i,x_i^T,i\in s\}$, where $s$ is a subset of $\{1,\cdots,n\}$ with the cardinality $k$ and $s^c$ is the complement of $s$ such that $s\cup s^c=\{1,\cdots,n\}$. Then we let $y_s$ be the column vector consisting of the entries of $y$ indexed by $s$ (the $i$-th element of $y_{s}$ is $y_j$ with $j$ being the $i$-th entry of $s$),
$X_s$ be the $k\times p$ matrix consisting of the rows of $X$ indexed by $s$,  $\widehat{\beta}_{s^c}$ be the RLS estimator \dref{rls} using the training data $\{y_i,X_i,i\in s^c\}$, namely,
\begin{align}
\widehat{\beta}_{s^c} = G_{s^c}^{-1} X_{s^c}^Ty_{s^c},~
G_{s^c}=X_{s^c}^TX_{s^c} +\sigma^2K^{-1},\label{rsc}
\end{align}
where $y_{s^c}$ and $X_{s^c}$ are constructed in a similar way as was done for $y_s$ and $X_s$.
With the above conventions and notations, the average prediction error over the validation data indexed by the set $s$ is defined by
\begingroup
\allowdisplaybreaks
\begin{align}
\label{pecv}
{\rm APE}_s
=\frac1k\|y_s - X_s\widehat{\beta}_{s^c}\|^2
=\frac1k\|Z_s^{-1}(y_s - X_s\widehat{\beta})\|^2,
\end{align}
\endgroup
where $\widehat{\beta}$ is the RLS estimator \dref{rls} using the data $\{y_i,x_i^T,$ $i=1,\cdots,n\}$, and
\begin{align}\label{ZsG}
Z_s=I_k-X_sG^{-1}X_s^T,~G=X^TX +\sigma^2K^{-1},
\end{align}
and the last identity follows from \dref{iden1} in the Appendix B.

Now we are ready to review the aforementioned CV methods. Clearly, for $k=1,\dots, n-1$, there are in total $\tbinom{n}{k}$ ways to split the data $\{y_i,x_i^T,$ $i=1,\cdots,n\}$,
where $\tbinom{n}{k}$ is the combination of $n$ objects taken $k$ at a time.
The LKOCV is to estimate $\eta$ by minimizing the average prediction errors over all the $\tbinom{n}{k}$ ways of splitting the data $\{y_i,x_i^T,$ $i=1,\cdots,n\}$:\begingroup
\allowdisplaybreaks
\begin{equation}\label{lkocv}
\begin{aligned}
{\rm LKOCV:}~~\widehat{\eta}_{\rm lkocv}
= \argmin_{\eta \in \Omega} \mathscr{C}_{\rm lkocv}(\eta),~~
\mathscr{C}_{\rm lkocv}(\eta)=\frac 1{\tbinom{n}{k}}\sum_{s\in \mathcal{S}}{\rm APE}_s,
\end{aligned}
\end{equation}
\endgroup
where $ \mathcal{S}$ is set consisting of $s$, each corresponding to one validation data obtained in one of the $\tbinom{n}{k}$ splitting ways.

Obviously, the LKOCV is computationally expensive to apply for large $n$ and moderately large $k$. To reduce the computational cost, a  special case of the LKOCV (\ref{lkocv}) with $k=1$ is often considered and called the LOOCV, which is also known as the predicted residual sums of squares (PRESS) \cite{Allen1974}:
\begin{equation}
\begin{aligned}
{\rm LOOCV:}~~\widehat{\eta}_{\rm loocv}
= \argmin_{\eta \in \Omega} \mathscr{C}_{\rm  loocv}(\eta),~
\mathscr{C}_{\rm  loocv}(\eta)
=\frac1n\sum_{i=1}^n\left(\frac{y_i-\widehat{y}_i}{1-M_{ii}} \right)^2,
\label{loocv}
\end{aligned}
\end{equation}
where $\widehat{y}_i$ is the $i$-th element of the predicted output $\widehat{y}=X\widehat{\beta}$ and $M_{ii}$ is the $i$-th diagonal element of
\begin{align}
M=X(X^T   X   + \sigma^2K^{-1})	^{-1}X^T \label{hat}.
\end{align}

To further reduce the computational cost, those different weights $1-M_{ii}$ for the prediction error in the LOOCV (\ref{loocv}) are replaced by their average $1-{\rm Tr}(M)/n$, leading to the GCV:
\begingroup
\allowdisplaybreaks
\begin{equation}
\begin{aligned}
{\rm GCV:}~~\widehat{\eta}_{\rm gcv}
&	= \argmin_{\eta \in \Omega} \mathscr{C}_{\rm gcv}(\eta)\\
\mathscr{C}_{\rm gcv}(\eta)
&=\frac1n\frac{\sum\limits_{i=1}^n\big(y_i-\widehat{y}_i\big)^2}
{\big(1-{\rm Tr}(M)/n\big)^2}
=\frac1n\frac{\|y-X\widehat{\beta}\|^2}
{\big(1-{\rm Tr}(M)/n\big)^2}. \label{gcv}
\end{aligned}
\end{equation}
\endgroup

The RFCV is another way to reduce the computation cost of the LKOCV.
Suppose that $k$ is a factor of $n$, namely, there exists an integer $r$ such that $n=kr$. Then the set $\{1,2,\cdots,n\}$ is divided into $r$ subsets
denoted by
\begin{align}
s_i=\{(i-1)k+1,\cdots,ik\},i=1,\cdots,r.\label{si}
\end{align}
Then the RFCV estimates  $\eta$ by
\begin{equation}\label{rfcv}
\begin{aligned}
{\rm RFCV:}~~\widehat{\eta}_{\rm rfcv}
= \argmin_{\eta \in \Omega} \mathscr{C}_{\rm rfcv}(\eta),~~
\mathscr{C}_{\rm rfcv}(\eta)
=\frac 1{r}\sum_{s\in \mathscr{S}}{\rm APE}_s,
\end{aligned}
\end{equation}
where $\mathscr{S}=\{s_i,i=1,\cdots,r\}$.
Obviously, the $n$-fold cross validation is exactly the LOOCV.

Lastly, the HOCV considers one validation data indexed by one $s\in\mathscr{S}$ in \dref{si} and estimates $\eta$ according to
\begin{equation}
\label{hocv}
\begin{aligned}
{\rm HOCV:}~~\widehat{\eta}_{\rm hocv}
= \argmin_{\eta \in \Omega} \mathscr{C}_{\rm hocv}(\eta),~~
\mathscr{C}_{\rm hocv}(\eta)
={\rm APE}_{s},
\end{aligned}
\end{equation}
where $s\in\mathscr{S}$ in \dref{si} and contains $k$ elements.

\section{Main Results}
\label{sec:main}

In this section, we study the a.o. of the hyper-parameter estimators reviewed in Section \ref{sec:cv}.

\subsection{Asymptotic Optimality of GCV Estimator for the RLS estimator \eqref{rls}}
\label{4.1}
We prove the a.o. of  the GCV estimator using Lemma \ref{ct} in the Appendix B.
The key idea
is to find an affine transform, irrespective of $\eta$, of the cost function $\mathscr{C}_{\rm gcv}(\eta)$ such that the transformed cost function converges to the limiting cost function $W(\eta)$.
\begin{thm}
	\label{thm1}
	Under Assumptions \ref{assum1}-\ref{ass2}, it holds that
	\begin{align}
	n^{2}\big(
	\mathscr{C}_{\rm gcv}(\eta) - \widehat{\sigma}^2(1+2p/n+3p^2/n^2)
	\big)/\sigma^4
	\xra{}W(\eta) \label{lgcv}
	\end{align}
	in probability  as $n\xra{}\infty$,
	where
	\begin{align}
	\widehat{\sigma}^2
	=(y^Ty-y^TX(X^T X)^{-1}X^Ty)/n\label{sigma}
	\end{align}
	is a consistent estimate for $\sigma^2$,
	and moreover, $\widehat{\eta}_{\rm gcv}$ is a.o..
\end{thm}

\begin{rem} It is worth to note that it was mentioned in \cite[p. 1355]{Li1985} that for the ridge regression estimator \eqref{rr}, the consistency of the GCV estimator was studied in \cite{Erdal1983} under the assumption that $p$ is fixed and $X^TX/n=O(1)$. Unfortunately, the authors are unable to find and access \cite{Erdal1983} and thus cannot make a detailed comparison with \cite{Erdal1983}. 
\end{rem}

\begin{rem} The a.o. of the GCV estimator in the sense of \eqref{aor} requires that $X^TX$ tends to almost singular as $n\xra{}\infty$ in \cite{Li1986}. In contrast, the a.o. of the GCV estimator requires that $X^TX$ converges to a positive definite as $n\xra{}\infty$.	
\end{rem}

\subsection{Asymptotic Optimality of other CV based Hyper-parameter Estimators  for the RLS estimator \eqref{rls}}
\label{4.2}
Since it is difficult to find an affine transform of the cost functions of the LOOCV, LKOCV, RFCV, and HOCV estimators as that used for the GCV estimator \eqref{gcv}, we show their a.o. 
to the first order
by exploring their first order optimality condition with the help of Lemma \ref{fdt} in the Appendix B.

For investigating the a.o. to the first order of these CV estimators except for the GCV estimator,  we first list some extra assumptions on the finite forth-moment of $\{\varepsilon_i,i=1,\dots, n\}$ and on the uniform boundedness of  $x_i^T,~i=1,\dots, n$.
\begin{assum} \label{ass3}
	\begin{enumerate}[1)]
		\item For $i=1,\cdots,n$, $E(|\varepsilon_i|^4)\leq c_3<\infty$.
		\item The vectors $x_i^T,~i=1,\dots, n$,  are uniformly bounded from above by a constant $c_4>0$, namely, $\sup_{n}\max_{i=1,\dots, n}\|x_i\|\leq c_4$.
		\item For $i=1,\cdots,m$, $\frac{\partial K(\eta)}{\partial \eta_i}$ is continuous over $\eta\in\Omega$. 
	\end{enumerate}
\end{assum}

\subsubsection{Asymptotic Optimality of LOOCV based Hyper-parameter Estimator  for the RLS estimator \eqref{rls}}

The a.o. to the first order of the  LOOCV estimator is given in the following theorem.
\begin{thm}
	\label{thm2}
	Under Assumptions \ref{assum1}-\ref{ass3}, it hold that 
	\begin{align}
	n^{2}\frac{\partial \mathscr{C}_{\rm loocv}(\eta)}{\partial K}
	\xra{}
	2\sigma^4
	K^{-1}
	\Sigma^{-1}
	K^{-1}
	(K -  \beta \beta ^T)
	K^{-1}
	\label{fdloocv}
	\end{align}
	in probability  as $n\xra{}\infty$ and moreover, $\widehat{\eta}_{\rm loocv}$ is a.o. to the first order.
\end{thm}

\subsubsection{Asymptotic Optimality of LKOCV Estimator}

For the LKOCV estimator \eqref{lkocv} with $k\geq 2$, especially $k\xra{}\infty$ as $n\xra{}\infty$,
the derivation of its a.o. requires more sophisticated analysis, which depends on the limit of the ratio between the number of samples contained in the validation data and the whole data, i.e., 
$
\lambda\eq \lim_{n\xra{}\infty}k/n.$

Moreover, we need one more assumption on the regression matrix $X$ to guarantee the a.o. to the first order of the LKOCV for the case $0< \lambda <1$.

\begin{assum} \label{assum4}
	When    $k\xra{}  \infty$ as $n\xra{}\infty$ and $0< \lambda <1$, the regression matrix $X$ satisfies 
	$
	X_s^TX_s / k \xra{}\Sigma~\mbox{as} ~k\xra{}\infty$
	for all $s\in \mathcal{S}$,
	where $\mathcal{S}$ was defined in \eqref{lkocv}.
\end{assum}	
\begin{rem}
	Assumptions similar to Assumption \ref{assum4} are used in \cite[Theorem 1]{Shao1993} and \cite[Assumption B]{Zhang1993} to prove asymptotic properties of the CV methods for model order selection of linear regression problems.
\end{rem}
Then the a.o. to the first order of the LKOCV is given in the following theorem.
\begin{thm}
	\label{thm3}
	Suppose Assumptions \ref{assum1}-\ref{ass3} hold.	
	
	1) If $\lambda=0$, then $\widehat{\eta}_{\rm lkocv} $ is a.o. to the first order.
	
	2) If $0< \lambda <1$ and Assumption \ref{assum4} holds, then $\widehat{\eta}_{\rm lkocv} $ is a.o. to the first order.
\end{thm}
\begin{rem}
	The statement $\lambda=0$ includes the cases: 1)  $k$ is  a fixed integer; 2) $k=O(n^\tau)$ with $0< \tau <1$, and etc..
\end{rem}

\subsubsection{Asymptotic Optimality of RFCV based Hyper-parameter Estimator for the RLS estimator \eqref{rls}}

The a.o. to the first order of the RFCV estimator \eqref{rfcv} depends on whether or not the $r$ in \eqref{si} tends to infinity as $n\xra{}\infty$.

\begin{thm}
	\label{thm4}
	Suppose Assumptions \ref{assum1}-\ref{assum4} hold.
	
	1) If $r \xra{}\infty$ as $n\xra{}\infty$, then $\widehat{\eta}_{\rm rfcv} $ is a.o. to the first order.
	
	2) If $r$ is a fixed positive integer,  then
	\begin{align*}
	n^{2}\frac{\partial \mathscr{C}_{\rm rfcv}(\eta)}{\partial K}
	&=
	\frac{2\sigma^4r^2}{(r - 1)^2}
	\big(K^{-1}\Sigma^{-1} K^{-1}
	-  K^{-1}\Sigma^{-1}K^{-1}\beta \beta ^TK^{-1}\big)\\
	&\hspace{5mm}
	+\frac{2\sigma^2r^2}{(r - 1)^2}
	\big(K^{-1}\Sigma^{-1}
	S'_n
	\Sigma^{-1} K^{-1}
	\big)
	+o_p(1)\nonumber
	\end{align*}
	where  $S'_n\eq X^T L' X/n=O_p(1)$ has zero mean and finite variance and the $(kl)$-element of $L'$ is defined as
	\begin{align}
	\label{li}
	\left\{
	\begin{array}{ll}
	0 &\mbox{ if } k=l \\
	\varepsilon_k\varepsilon_l & \mbox{ if }  k\in s \mbox{ and } l\in s \mbox{ for some $s\in \mathscr{S}$} \mbox{ but } k\neq l\\
	-1/(r-1) \varepsilon_k\varepsilon_l  & \mbox{ otherwise}.
	\end{array}\right.
	\end{align}
	If $S'_n\xra{}0$ in probability as $n\xra{}\infty$, then $\widehat{\eta}_{\rm rfcv} $ is a.o. to the first order, and otherwise, $\widehat{\eta}_{\rm rfcv} $ is in general not a.o. to the first order. 
\end{thm}

\begin{rem}
	Define the matrix $A'$ with the $(kl)$-element corresponding to $L'$ in \dref{li} as
	\begin{align*}
	\left\{
	\begin{array}{ll}
	0 &\mbox{ if } k=l \\
	1 & \mbox{ if }  k\in s \mbox{ and } l\in s \mbox{ for some $s\in \mathscr{S}$} \mbox{ but } k\neq l\\
	-1/(r-1)  & \mbox{ otherwise}.
	\end{array}\right.
	\end{align*}
	If there exists some $ij$ such that $|X_{ki}|>\alpha>0$ and $|X_{lj}|>\beta>0$ for all $1\leq k,l \leq n$, then under Assumption \ref{ass3} the variance of the $(ij)$-element of $S'_n$ is
	$\frac{\sigma^4}{n^2}\sum_{k\neq l}
	( X^T_{ik} A'_{kl}X_{lj} )^2>
	\frac{\sigma^4\alpha^2\beta^2(n^2-n)}{n^2(r-1)^2}
	\geq\frac{\sigma^4\alpha^2\beta^2}{2(r-1)^2}>0.$
	This means that $S'_n\nrightarrow{}0$ in probability  as $n\xra{}\infty$ and hence   $\widehat{\eta}_{\rm rfcv} $ is  not a.o. to the first order for this case. 
\end{rem}

\subsubsection{Asymptotic Properties of HOCV based Hyper-parameter estimator  for the RLS estimator \eqref{rls}}

Lastly, we consider the HOCV estimator \dref{hocv}.
\begin{thm}
	\label{thm5}
	Suppose Assumptions \ref{assum1}-\ref{assum4} hold.
	If $r$ is a fixed positive integer, then
	\begin{align*}
	n^{3/2}\frac{\partial \mathscr{C}_{\rm hocv}(\eta)}{\partial K}
	=-\frac{2\sigma^2r^3}{(r - 1)^2}
	K^{-1}\Sigma^{-1}
	\left(
	\frac{1}{\sqrt{n}}
	\big(X_s^T\varepsilon_s - r^{-1}X^T\varepsilon\big)\right)
	\beta K^{-1}
	+o_p(1).
	\end{align*}
	Therefore,  $\widehat{\eta}_{\rm hocv} $ is in general not a.o. to the first order. 
\end{thm}

\begin{rem} When $\{x_i,1\leq i \leq n\}$ are random, the results derived in this section  still hold if 1) $\{x_i,1\leq i \leq n\}$ are uncorrelated to $\varepsilon$; 2) $X^TX/n$ converges to a positive definite matrix $\Sigma$ in probability as $n\xra{}\infty$. 
\end{rem}

\subsection{Asymptotic Optimality of CV based Hyperparameter Estimators for the ridge regression estimator \dref{rr}}

In this section, we show that for the ridge regression estimator \eqref{rr}, all CV but not HOCV based hyper-parameter estimators studied above can be a.o. under some mild assumptions. 
\begin{thm}
	\label{thm6}
	Consider the ridge regression estimator \dref{rr}.
	Suppose that Assumptions \ref{assum1}-\ref{ass2} hold.
	Thus, there hold that
	\begin{enumerate}[1)]
		\item $\widehat{\eta}_{\rm gcv}$ is a.o.;
		\item $\widehat{\eta}_{\rm loocv}$ is a.o. if in addition Assumption \ref{ass3} holds;
		
		\item 
		$\widehat{\eta}_{\rm lkocv} $ is a.o.  when $\lambda=0$ if  in addition  Assumption \ref{ass3} holds;
		
		\item 
		$\widehat{\eta}_{\rm lkocv} $ is a.o.  when $0<\lambda<1$ if  in addition  Assumptions \ref{ass3} and \ref{assum4} hold;
		
		\item 	  $\widehat{\eta}_{\rm rfcv} $ is a.o.  when $r \xra{}\infty$ if  in addition  Assumptions \ref{ass3} and \ref{assum4} hold;.
		
		\item 	  $\widehat{\eta}_{\rm rfcv} $ is in general not a.o.  when $r$ is fixed even if  in addition  Assumptions \ref{ass3} and \ref{assum4} hold;
		
		\item 	  $\widehat{\eta}_{\rm hocv} $ is in general not a.o.  when $r$ is fixed even if  in addition  Assumptions \ref{ass3} and \ref{assum4} hold.
		
	\end{enumerate}

\end{thm}

	\section{Numerical Simulation}
	
	In this section, we illustrate our theoretical results
	by  Monte Carlo numerical simulations.

	\subsection{Data Bank}
	
	For the linear regression model \eqref{lrm}, 
	we generate three data banks, corresponding to three different ways to generate the regression matrix $X$, which are described as follows:
	\begin{enumerate}
		\item{D1}: well-conditioned $X$, $X^TX$ is approximately equal to $nI_p$;
		
		\item{D2}: ill-conditioned $X$, the condition number of $X^TX$ is around $10^{7}$;
		
		\item{D3}: unbounded $
		X=\begin{pmatrix}
		\sqrt{n}I_p\\
		0
		\end{pmatrix}
		$ with respective to $n$,
		which
		satisfies Assumption \ref{assum1} ($X^TX=nI_p$) but contradicts Assumption \ref{assum4}.
	\end{enumerate}
	
	Each data bank contains 1000 data sets with $n=300$ and 1000 data sets with $n=7800$. To generate each data set, the elements of $\beta\in\mathbb R^{200}$ are chosen to decay exponentially, i.e., $\beta_i=O(\mu^i)$ with $0\leq \mu <1$ for $1\leq i \leq p$, which has wide applications in the area of control and systems, e.g.. \cite{Ljung1999}. The noise $\varepsilon$ is further assumed to be Gaussian with the variance such that the signal-to-noise ratio  (SNR), i.e., the ratio between the variance of the signal $X\beta$ and the noise $\varepsilon$,  is uniformly distributed over $[1,10]$.

	\subsection{Simulation Setup}

	The kernel matrix $K(\eta)$ is parameterized by the Tuned/Correlated (TC) kernel defined as
	$K_{ij}(\eta) = c\mu^{\max(i,j)},~
	\eta = [c,\mu]\in\Omega=\{c\geq 0,0\leq \mu \leq1\}$,
	which can well characterize the exponential decay of the $\beta$ \cite{Chen2012}.
	We use the risk \dref{ophy_risk}, GCV \dref{gcv}, LOOCV \dref{loocv}, 2-fold, 5-fold, and 100-fold (for $N=300$)/3900-fold (for $N=7800$) \dref{rfcv}, and 2-hocv and 5-hocv \dref{hocv} estimators to tune the hyperparameter $\eta$ of the TC kernel $K(\eta)$, then we calculate their corresponding RLS $\widehat{\beta}(K(\widehat\eta))$, and finally we calculate the measure of fit
	\begin{align*} \mbox{\rm Fit} =100\times \left( 1 -
	\frac{\|\widehat{\beta} -
		\beta \|}{\|\beta -\bar{\beta}\|}\right),~~\bar{\beta}=\frac{1}{200}\sum_{i=1}^{200}\beta_i
	\end{align*}
	to evaluate how good these estimators are.
	Here the risk estimator \dref{ophy_risk} is impractical but provides a reference for other estimators and the LKOCV with $k>1$ is not illustrated due to the high computational cost.

	\subsection{Simulation results}
	The boxplots of the fits for the three data banks D1, D2, and D3 are shown in Figs.
	\ref{f1}--\ref{f3}, respectively, where the numbers at the bottom of these figures indicate for each hyperparameter estimator the number of the fits lower than the lowest value of the y-axis out of the 1000 data sets,
	and the  corresponding average fits are given in Table \ref{tab1}.
	
	\begin{figure*}[!h]
		\centering
		\hspace{-20mm}
		\mbox{\includegraphics[scale=0.4]{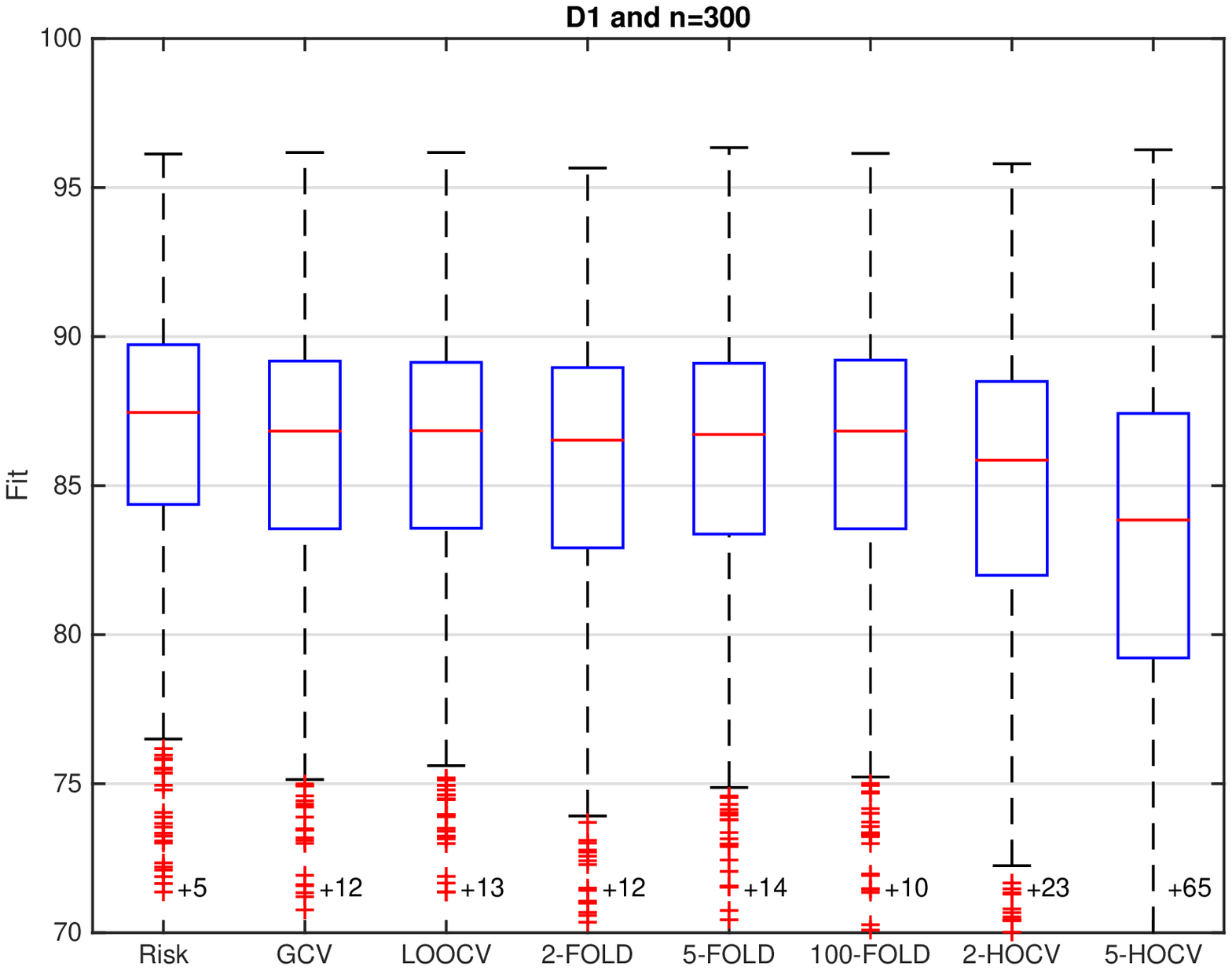}
			\hspace{-8mm}
			\includegraphics[scale=0.4]{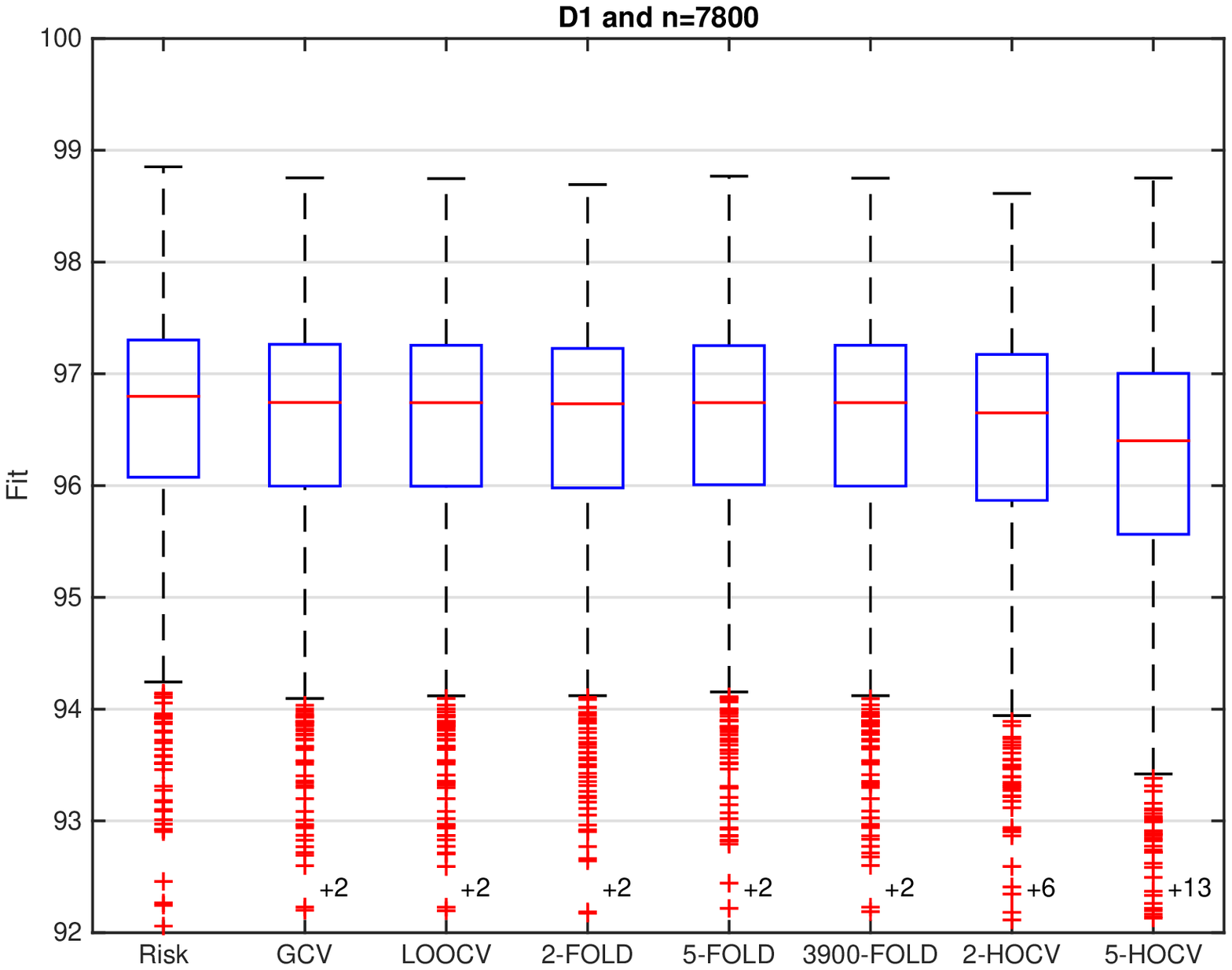}}
		\vspace{-3ex}
		\caption{Boxplot of the 1000 fits for the data bank D1: $N=300$ (left) and $N=7800$ (right).}
		\label{f1}
	\end{figure*}
	\begin{figure*}[!h]
		\centering
		\hspace{-20mm}
		\mbox{\includegraphics[scale=0.4]{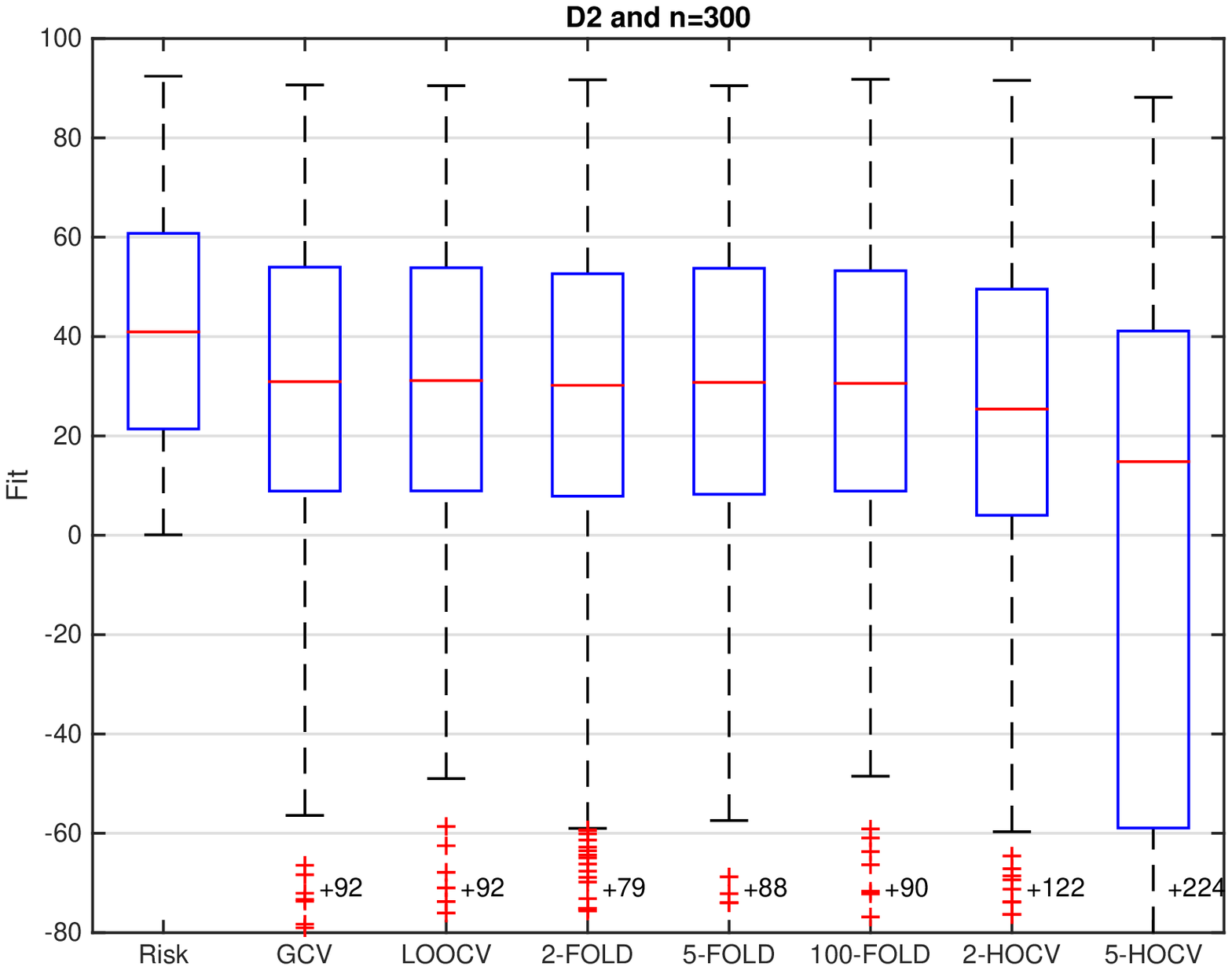}
			\hspace{-8mm}
			\includegraphics[scale=0.4]{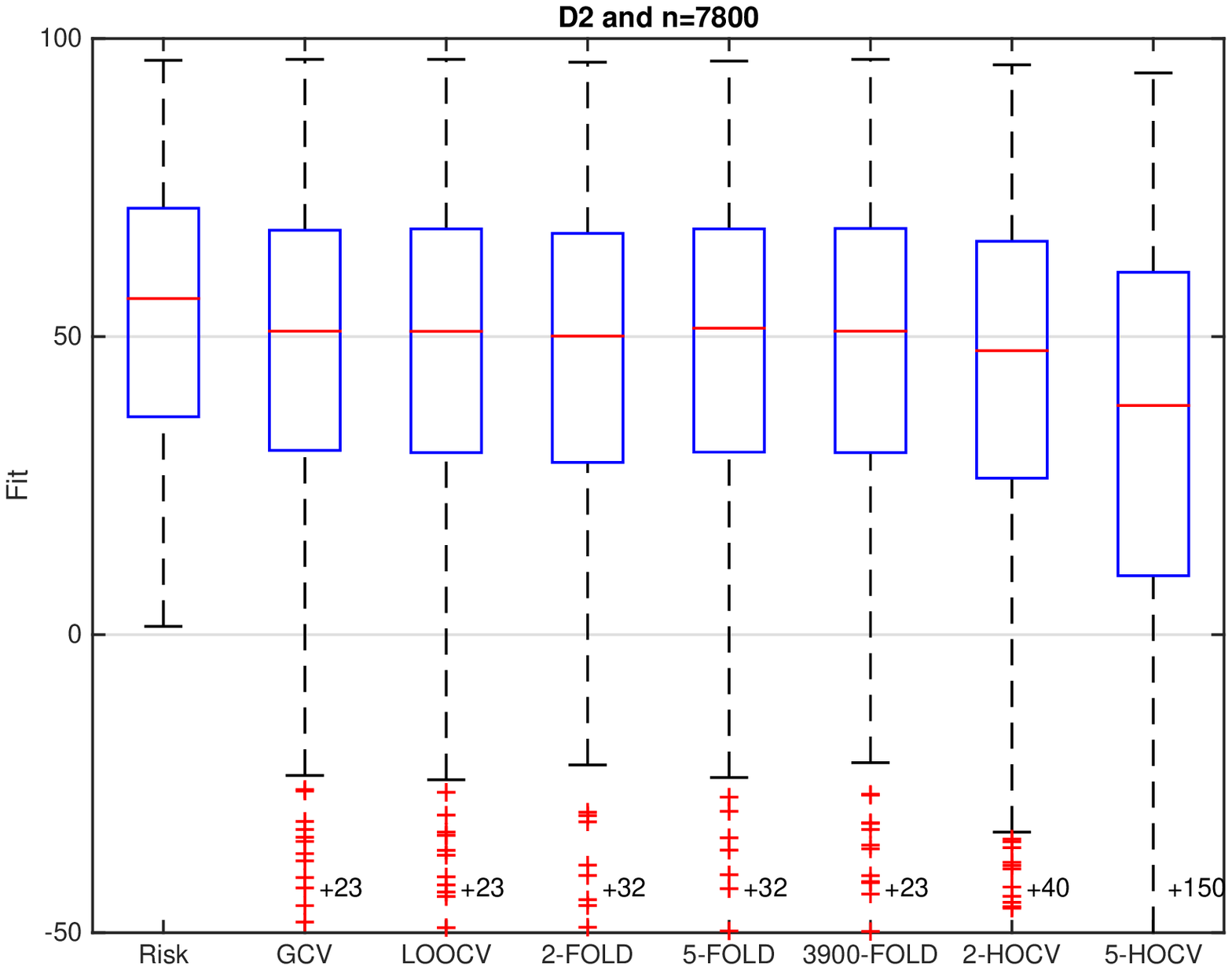}}
		\vspace{-3ex}
		\caption{Boxplot of the 1000 fits for the data bank D2: $N=300$ (left) and $N=7800$ (right).}
		\label{f2}
	\end{figure*}
	\begin{figure*}[!h]
		\centering
		\hspace{-20mm}
		\mbox{\includegraphics[scale=0.4]{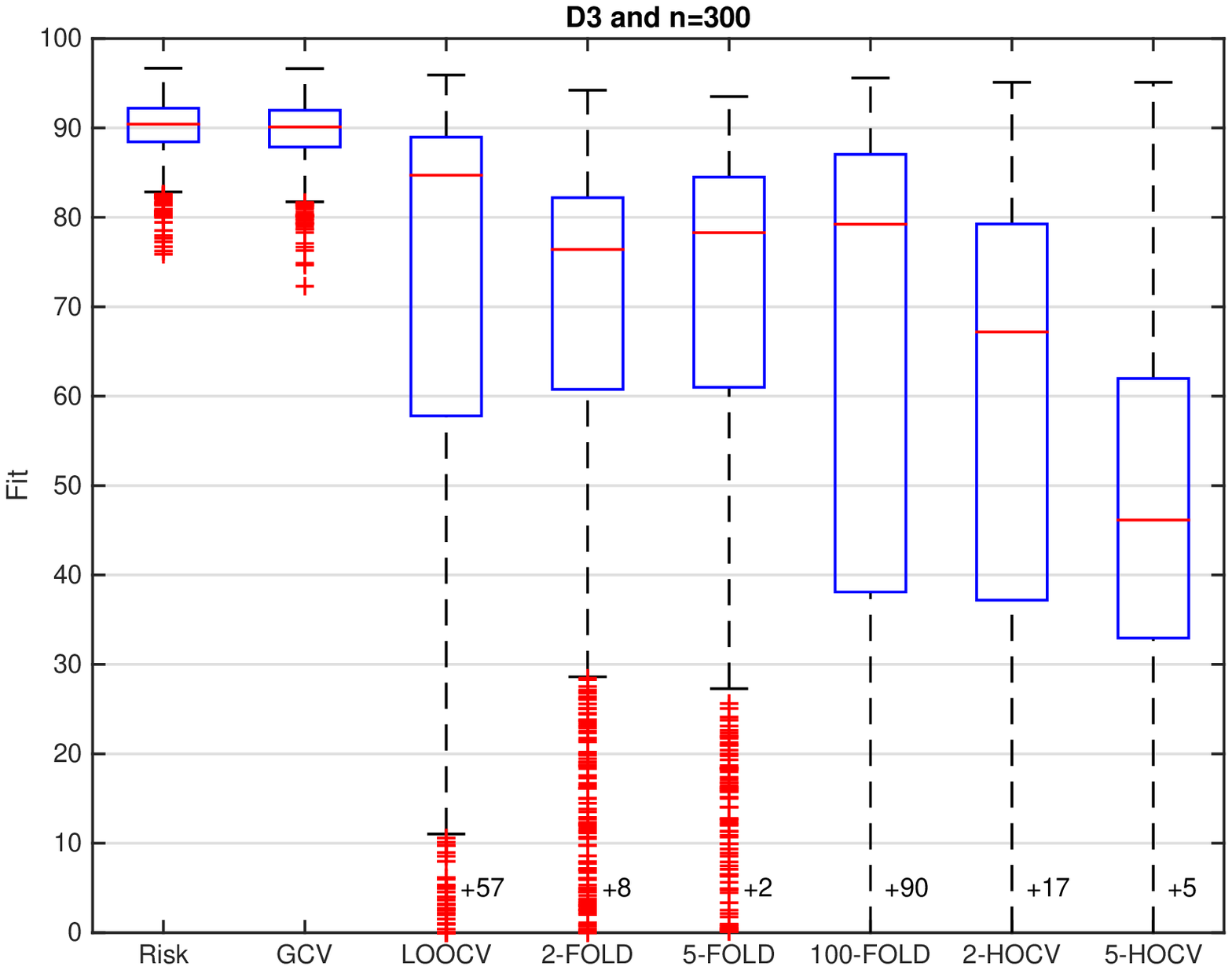}
			\hspace{-8mm}
			\includegraphics[scale=0.4]{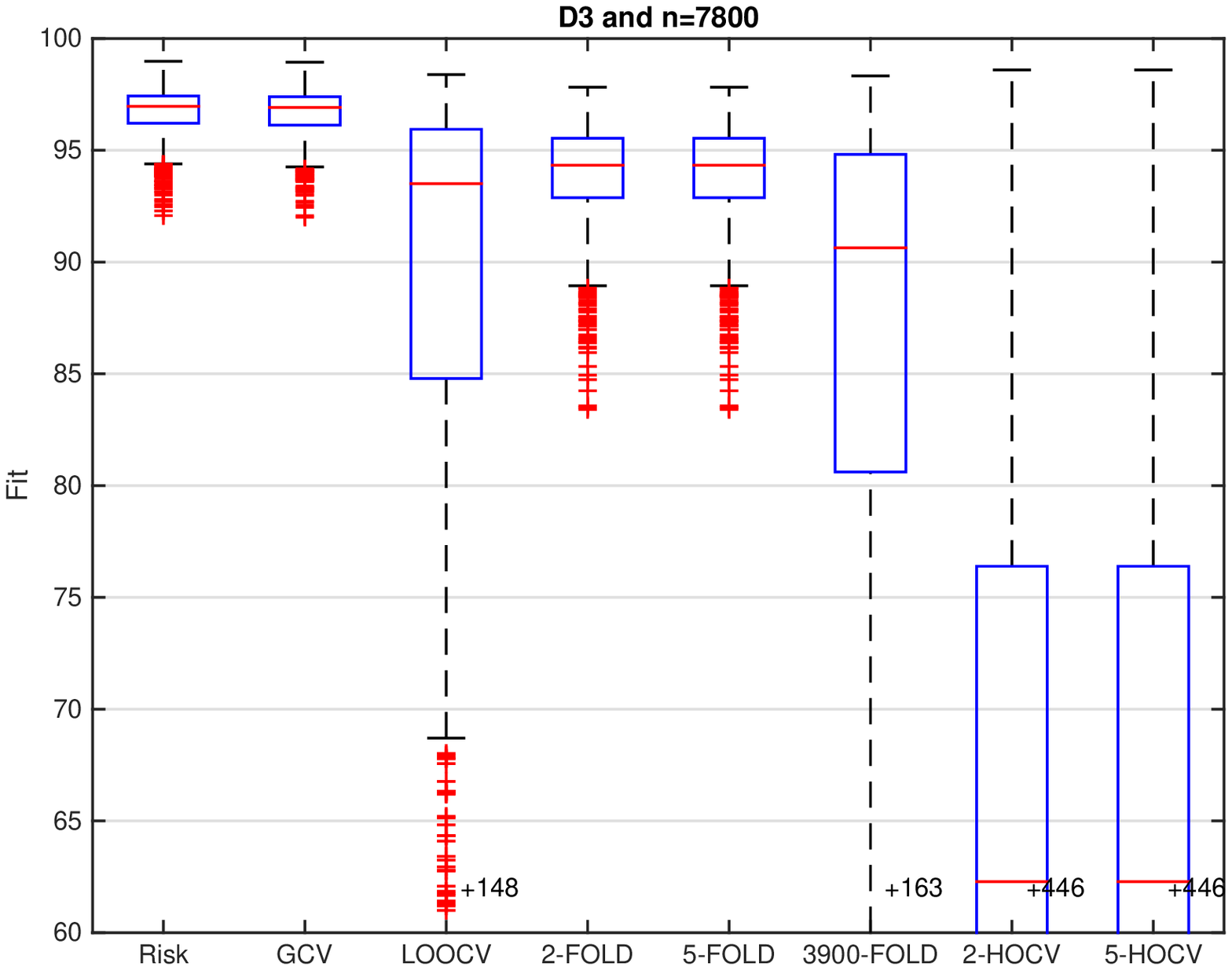}}
		\vspace{-3ex}
		\caption{Boxplot of the 1000 fits for the data bank D3: $N=300$ (left) and $N=7800$ (right).}
		\label{f3}
	\end{figure*}
	\begin{table}[!ht]
		\centering
		{\footnotesize 
		\caption{Average fits for the three data banks D1, D2, and D3.}
		\vspace{1.5ex}
		\begin{tabu} to 1
			\textwidth{X[0.8,l]X[0.6,r]X[0.6,r]X[0.6,r]X[0.6,r]X[0.6,r]X[1,r]X[0.7,r]X[0.7,r]}
			\hline
			&   \mbox{\scriptsize Risk}    &   \mbox{\scriptsize GCV}  &    \mbox{\scriptsize LOOCV}  &   \mbox{\scriptsize 2-FOLD}  &    \mbox{\scriptsize 5-FOLD}   &   \mbox{\scriptsize 100-FOLD} \mbox{\scriptsize /3900-FOLD}     &   \mbox{\scriptsize 2-HOCV} &   \mbox{\scriptsize 5-HOCV}  \\  \hline
			&&&&\mbox{The data bank D1} \\ \hline
			\mbox{$n\!=\!300$}    & 86.69 	&	85.93 	&	85.93 	&	85.56 	&	85.71 	&	85.92 	&	84.64 	&	82.28 	\\  
			\mbox{$n\!=\!7800$}    &   96.56 	&	96.49 	&	96.49 	&	96.46 	&	96.49 	&	96.49 	&	96.39 	&	96.11    \\  \hline
			
			&&&&\mbox{The data bank D2} \\ \hline
			\mbox{$n\!=\!300$}    &   41.63 	&	-11.42 	&	-10.58 	&	-13.17 	&	-14.79 	&	-9.56 	&	-37.53 	&	\mbox{-225.78}    \\
			\mbox{$n\!=\!7800$}    &  53.58 	&	43.68 	&	43.61 	&	39.69 	&	41.63 	&	43.56 	&	37.22 	&	-16.02      \\  \hline
			
			&&&&\mbox{The data bank D3} \\ \hline
			\mbox{$n\!=\!300$}     &   89.94 	&	89.53 	&	69.34 	&	66.94 	&	72.80 	&	62.19 	&	56.93 	&	47.30       \\
			\mbox{$n\!=\!7800$}    &   96.70 	&	96.64 	&	81.40 	&	93.89 	&	93.89 	&	78.55 	&	63.25 	&	63.25      \\  \hline
		\end{tabu}}
		\label{tab1}
	\end{table}
	\subsection{Findings}
	
	Firstly, for all tested cases and in terms of average accuracy and
	robustness, the risk estimator  (not implementable
	in practice) performs best.
	
	Secondly, for the data banks D1 and D2 (well/ill-conditioned $X^TX$), excluding the risk estimator, the GCV and LOOCV estimators perform almost identically and are better than other estimators;
	the RFCV estimators are gradually improved as the number $r$ increases;
	the HOCV estimators perform worst and the performance deteriorates as $r$ increases.

	Thirdly, for the data bank  D3 (unbounded regressors), the GCV estimator is very close to the risk estimator but the LOOCV, RFCV, and HOCV are significantly 
	worse. This means that the requirement on the boundedness imposed in the second term of Assumption \ref{ass3} is necessary for the LOOCV and RFCV estimators.
	The simulation results on the 2-fold, 5-fold,  2-hocv, and 5-hocv  for $n=7800$ are meaningless  since the cost functions are constant and the estimated hyperparameters are exactly the preset initial values of optimization algorithms.

	The numerical simulations verify our theoretical results:
	\begin{enumerate}[1)]
		\item the GCV estimator is a.o. without the boundedness assumption on the regressors;
		\item the LOOCV estimator is a.o. by further imposing the boundedness assumption on the regressors;
		\item the RFCV estimator is a.o.  as $r$ increases to infinity by further imposing the boundedness assumption on the regressors, and is just a little bit worse than the GCV and LOOCV for fixed $r$ and well conditioned $X^TX$;
		\item the HOCV estimator is in general not a.o..
	\end{enumerate}

\section{Conclusions}
\label{sec7}

Since the CV methods are widely used for the hyper-parameter estimation for the regularized linear least squares regression problems, it is interesting and important to know the asymptotic optimality of CV based hyper-parameter estimators. In this paper, we focused on this issue and show that for regularized linear least squares regression problems with a fixed number of regression parameters, the GCV, LKOCV and RFCV estimators are asymptotically optimal under some mild assumptions, but the HOCV estimator is in general not. Our theoretical results provide a reference for practitioners who are using the CV methods and interested in the asymptotic optimality of the CV based hyper-parameter estimators for the regularized linear least squares regression problems.  

\section{Appendix}  
Appendix A contains the proof of the results in the paper and Appendix B 
contains the technical lemmas for the proofs of our theoretical results in Appendix A.
\begin{appendix}
	\section{Proofs of results}
	
	\renewcommand{\thesection}{A}
	\setcounter{rem}{0}
	\renewcommand{\therem}{A\arabic{rem}}
	\setcounter{equation}{0}
	\renewcommand{\theequation}{A\arabic{equation}}
	\subsection{Proof of Proposition \ref{thm21}}
	We first consider the proof of the one based on inefficiency, i.e.,  $L_n(\widehat{\eta}_n)/\inf\limits_{\eta\geq 0} L_n(\eta)\xra{}1$.
	It follows that
	\begin{align}
	\nonumber
	L_n(\eta)&=	n^{-1}\eta^2\beta^T G(\eta)^{-1}X^TX G(\eta)^{-1} \beta
	\\
	\nonumber
	&\hspace{4mm}
	+
	n^{-1}\varepsilon^T X G(\eta)^{-1} X^TX  G(\eta)^{-1} X^T \varepsilon
	\\
	\nonumber
	&\hspace{4mm}
	-2 n^{-1}\eta \beta^TG(\eta)^{-1}  X^TX G(\eta)^{-1} X^T \varepsilon\\
	\nonumber
	&=L'_n(\eta)+L''_n(\eta)+L'''_n(\eta)
	\end{align}
	where $G(\eta) =  X^TX + \eta I_p$ and $L'_n(\eta)$, $L''_n(\eta)$, $L'''_n(\eta)$ are used to denote the three terms in the first equation in order.
	Define  $\eta_n^\natural=\argmin_{\eta\geq 0} L_n(\eta)$ and
	let us discuss the possible minima of $L_n(\eta)$ by computing the different values of $\eta_n^\natural$.
	
	
	{\it Case 1:} $\eta_n^\natural = 0$.
	
	In this case, $G(\eta_n^\natural ) = X^TX$ and so $L_n(\eta_n^\natural)=0+n^{-1}\varepsilon^T X  (X^TX )^{-1} X^T \varepsilon+0=n^{-1}\varepsilon^T X  (X^TX )^{-1} X^T \varepsilon=O_p(1/n)>0$ since $\varepsilon^T X  (X^TX )^{-1} X^T \varepsilon=O_p(1)$ by Lemma \ref{lem2}.
	
	{\it Case 2:} $\eta_n^\natural = \infty$.
	
	It follows that $G(\eta_n^\natural )^{-1} = (X^TX)^{-1} - (X^TX)^{-1}  (I_p/\eta_n^\natural  + (X^TX)^{-1})^{-1} (X^TX)^{-1}=0$
	and $\eta_n^\natural  G(\eta_n^\natural )^{-1} = I_p$.
	Thus one derives that $L_n(\eta_n^\natural)=n^{-1}\beta^T X^TX  \beta +0+0=n^{-1}\beta^T X^TX  \beta = O(1)>0$.
	
	{\it Case 3:} $0< \eta_n^\natural < \infty$.
	
	One has $\eta_n^\natural = O_p(1)$, $G(\eta_n^\natural )^{-1}=O_p(1/n)$, $X^TX =O(n)$, and $\beta = O(1)$.
	As a result, $L'_n(\eta_n^\natural)=O(1/n^2)$, $L''_n(\eta_n^\natural)=n^{-1}\varepsilon^T X G(\eta_n^\natural)^{-1} X^TX  G(\eta_n^\natural)^{-1} X^T\varepsilon=O_p(1/n)$, and $L'''_n(\eta_n^\natural)=O_p(1/n^{3/2})$ since $X^T\varepsilon=O_p(\sqrt{n})$ and $n^{-1}\varepsilon^T X  X^T \varepsilon=O_p(1)$ by Lemma \ref{lem2}.
	
	Comparing the possible values of $L_n(\eta)$ for the above three cases, one finds that the value $\min_{\eta\geq 0}L_n(\eta)$ is either
	$n^{-1}\varepsilon^T X  (X^TX )^{-1} X^T \varepsilon$ or asymptotically $n^{-1}\varepsilon^T X G(\eta_n^\natural)^{-1} X^TX  G(\eta_n^\natural)^{-1} X^T\varepsilon$ for large $n$.
	Note that $G(\eta_n^\natural )^{-1} = (X^TX)^{-1}+\eta_n^\natural  G(\eta_n^\natural )^{-1}(X^TX)^{-1}=(X^TX)^{-1}+ O_p(1/n)(X^TX)^{-1}$.
	This means that
	$$n^{-1}\varepsilon^T X G(\eta_n^\natural)^{-1} X^TX  G(\eta_n^\natural)^{-1} X^T=n^{-1}\varepsilon^T X  (X^TX )^{-1} X^T \varepsilon+O_p(1/n^2).$$
	Therefore, it can be concluded that 
	\begin{align*}
	\min_{\eta\geq 0}L_n(\eta)
	=\left\{
	\begin{array}{ll}
	n^{-1}\varepsilon^T X  (X^TX )^{-1} X^T \varepsilon,&~if~\eta_n^\natural =0\\
	n^{-1}\varepsilon^T X  (X^TX )^{-1} X^T \varepsilon+O_p(1/n^2),&~if~0< \eta_n^\natural < \infty.
	\end{array}
	\right.
	\end{align*}

	When $\widehat{\eta}_n$ is bounded in probability, i.e., $\widehat{\eta}_n=O_p(1)$, 
	using the same analysis for Case 3 ($0< \eta_n^\natural < \infty$) above can similarly derive
	$
	L_n(\widehat{\eta}_n)=
	n^{-1}\varepsilon^T X  (X^TX )^{-1} X^T \varepsilon+O_p(1/n^2).
	$
	Thus it proves that
	$L_n(\widehat{\eta}_n)/\inf\limits_{\eta\geq 0} L_n(\eta)\xra{}1$
	in probability as $n\xra{}\infty$.
	
	
	Then we  consider the proof of the one based on the expectation inefficiency, i.e.,  $\mathcal{I}_n=E L_n(\widehat{\eta}_n)/\inf\limits_{\eta\geq 0} EL_n(\eta) \xra{}1$. 
	The proof is similar  to the previous case. It is straightforward that
	\begin{align*}
	R_n(\eta) = EL_n(\eta)
	&=n^{-1}\eta^2\beta^TG(\eta)^{-1}X^TX G(\eta)^{-1}\beta\\
	\nonumber
	&\hspace{5mm}
	+n^{-1}\sigma^2 {\rm Tr}\big(G(\eta)^{-1}X^TX G(\eta)^{-1}X^TX \big) \\
	&=R'_n(\eta)  + R''_n(\eta)
	\end{align*}
	where $G(\eta) = X^T   X   + \eta I_p$ and ${\rm Tr}(\cdot)$ denotes the trace of a matrix.
	Let us compute the minimum of $L_n(\eta)$ by considering the possible values of $\eta_n^*$.
	
	
	{\it Case 1:} $\eta_n^* = 0$.
	
	In this case, $G(\eta_n^* ) = X^TX$ and so $R_n(\eta_n^*)=0+n^{-1}p\sigma^2$.
	
	{\it Case 2:} $\eta_n^* = \infty$.
	
	It follows that $G(\eta_n^* )^{-1} =0$
	and $\eta_n^*  G(\eta_n^* )^{-1} = I_p$.
	Thus one derives that $R_n(\eta_n^*)=n^{-1}\beta^T X^TX  \beta +0+0=n^{-1}\beta^T X^TX  \beta = O(1)>0$.
	
	{\it Case 3:} $0< \eta_n^* < \infty$.
	
	One has $\eta_n^* = O(1)$, $G(\eta_n^* )^{-1}=O(1/n)$, $X^TX =O(n)$, and $\beta = O(1)$.
	As a result, $R'_n(\eta_n^*)=O(1/n^2)$ and
	\begin{align}
	\!R''_n(\eta_n^*)\!=\underbrace{n^{-1}\sigma^2p}_{O\left(\frac1n\right)}
	-\underbrace{n^{-1}2\sigma^2\eta_n^* {\rm Tr}\big(G(\eta_n^*)^{-1} \big)}_{O\left(\frac1{n^{2}}\right)}
	+\underbrace{n^{-1}\sigma^2(\eta_n^*)^2 {\rm Tr}\big(G(\eta_n^*)^{-2}  \big)}_{O\left(\frac1{n^{3}}\right)}. \label{r2}
	\end{align}
	Therefore, one can  conclude that 
	\begin{align*}
	\min_{\eta\geq 0}L_n(\eta)
	=\left\{
	\begin{array}{ll}
	n^{-1}p\sigma^2,&~if~\eta_n^* =0\\
	n^{-1}p\sigma^2+O(1/n^2),&~if~0< \eta_n^* < \infty.
	\end{array}
	\right.
	\end{align*}
	For the case that  $\widehat{\eta}_n$ is bounded in probability, its risk $R_n(\widehat{\eta}_n)=n^{-1}p\sigma^2+O_p(1/n^2)$  using the same analysis of Case 3.
	Thus it yields that $\mathcal{I}_n=\frac{R_n(\widehat{\eta}_n)}{\inf\limits_{\eta\geq 0} R_n(\eta)} \xra{}1$ as $n\xra{}\infty$.
	
	\subsection{Proof of Proposition \ref{prop3}}
	It follows that
	\begin{align*}
	n^{2}  \left( R_n(\eta) - \sigma^2p/n \right)
	&=\eta^2\beta^TnG(\eta)^{-1}X^TX G(\eta)^{-1}\beta\\
	\nonumber
	&\hspace{5mm}
	-2n\sigma^2\eta {\rm Tr}\big(G(\eta)^{-1} \big)
	+n\sigma^2\eta^2 {\rm Tr}\big(G(\eta)^{-2}  \big).
	\end{align*}
	Clearly, $\eta^*_n=\argmin_{\eta\geq 0} n^{2}  \left( R_n(\eta) - \sigma^2p/n \right)$ since $\sigma^2p/n$ dose not depend on $\eta$.
	For any $0<\eta <\infty$, following from \eqref{lim1},
	one has $nG(\eta)^{-1}\xra{}\Sigma^{-1}$ as $n\xra{}\infty$ and hence
	$n^{2}  \left( R_n(\eta) - \sigma^2p/n \right)\xra{}\beta^T\Sigma^{-1}\beta \eta^2
	-
	2\sigma^2{\rm Tr}\big(\Sigma^{-1} \big)\eta$
	as $n\xra{}\infty$. Since the convergence is also uniformly in a compact set of the point $\sigma^2{\rm Tr}\big(\Sigma^{-1} \big)/(\beta^T\Sigma^{-1}\beta)$,
	we have 
	$
	\eta_n^*\xra{}\eta^*
	=\frac{\sigma^2{\rm Tr}\big(\Sigma^{-1} \big)}{\beta^T\Sigma^{-1}\beta}
	=\argmin_{\eta\geq 0}
	\beta^T\Sigma^{-1}\beta \eta^2$
	$-
	2\sigma^2{\rm Tr}\big(\Sigma^{-1} \big)\eta$ as $n\xra{}\infty$.
	
	On the other hand, note that the set $\psi^*$ defined in \dref{rootset2} is also the single point $\frac{\sigma^2{\rm Tr}\big(\Sigma^{-1} \big)}{\beta^T\Sigma^{-1}\beta}$.
	Similarly, one can prove that  $\psi_n^*\xra{}\psi^*$ as $n\xra{}\infty$.
	Thus, Definition \ref{defn2} is equivalent to Definition \ref{defn1} for the ridge regression estimator \dref{rr}.
	
	\subsection{Proof of Theorem \ref{thm1}}
	
	The proof for \eqref{sigma} can be found in \dref{nid1} of Lemma \ref{lem4}. So we only prove \eqref{lgcv} below.
	
	Using the Taylor expansion $1/(1-x)^2=1+2x+3x^2+O(x^3)
	$ around $x=0$ and noting from the definition of $M$ in \eqref{hat} and Assumption \ref{assum1} that ${\rm Tr}(M)/n=O(1/n)$, one can obtain the Taylor expansion for $\mathscr{C}_{\rm gcv}(K)
	$:
	\begin{align}
	\label{decom}
	\mathscr{C}_{\rm gcv}(K)
	&= \frac{\|y - X\widehat{\beta} \|^2}
	{n\big(1 - {\rm Tr}(M)/n\big)^2}\\
	&=\frac{\|y - X\widehat{\beta} \|^2}{n}
	\Big( 1+\frac{2{\rm Tr}(M)}{n}  +\frac{3({\rm Tr}(M))^2}{n^2} +O\Big(\frac1{n^3}\Big) \Big).\nonumber
	\end{align}
	Now we analyze the limit of each term of the above equation.
	
	Firstly, by using \dref{nid1} and \eqref{lim1} we have as $n\xra{}\infty$
	\begin{align}
	\label{de1}
	n^{2}\Big(\frac{\|y - X\widehat{\beta} \|^2}{n}
	-\widehat{\sigma}^2 \Big)
	&=\sigma^4n \widehat{\beta}_{\rm LS}^T
	K^{-1}G^{-1}X^TXG^{-1}K^{-1}
	\widehat{\beta}_{\rm LS}\\
	\nonumber
	&\xra{}\sigma^4\beta^T K^{-1}\Sigma^{-1}K^{-1}\beta.
	\end{align}
	in probability.
	
	Secondly, by using \dref{nid1} and noting 
	$n({\rm Tr}(M) - p)
	\xra{}-\sigma^2{\rm Tr}\big(\Sigma^{-1}K^{-1}\big)
	$ and $\widehat{\sigma}^2\xra{}\sigma^2,$	
	we have
	\begingroup
	\allowdisplaybreaks
	\begin{align}
	\label{de2}
	n^{2}\Big(&\frac{\|y - X\widehat{\beta} \|^2}{n} \frac{2{\rm Tr}(M)}{n}
	-\widehat{\sigma}^2  \frac{2p}{n}\Big)\\
	\nonumber
	&=n^{2}\Big( \Big(\widehat{\sigma}^2 +O_p\Big(\frac1{n^{2}}\Big) \Big) \frac{2{\rm Tr}(M)}{n}
	-\widehat{\sigma}^2  \frac{2p}{n}\Big)\\
	\nonumber
	&=n^{2}\Big( \frac{2\widehat{\sigma}^2}{n}  \big( {\rm Tr}(M) -p \big)
	+O_p\Big(\frac1{n^{3}}\Big)\Big)\\
	&\xra{}-2\sigma^4{\rm Tr}(\Sigma^{-1}K^{-1})\nonumber
	\end{align}
	\endgroup
	in probability.
	
	Lastly, by similar arguments as the above two steps, we have
	\begin{align}
	\label{de3}
	&\hspace{-5mm}n^{2}\Big(\frac{\|y - X\widehat{\beta} \|^2}{n} \frac{3  \big({\rm Tr}(M)\big)^2}{n^2}
	-\widehat{\sigma}^2  \frac{3p^2}{n^2}\Big)\\
	\nonumber
	&=n^{2}\Big( \big(\widehat{\sigma}^2 +O_p(1/n^{2}) \big)  \frac{3  \big({\rm Tr}(M)\big)^2}{n^2}
	-\widehat{\sigma}^2  \frac{3p^2}{n^2}\Big)\\
	\nonumber
	&=n^{2}\Big( \frac{3\widehat{\sigma}^2}{n^2}  \big(  \big({\rm Tr}(M)\big)^2 -p^2 \big)
	+O_p(1/n^{4})\Big)\\
	\nonumber
	&\xra{}0.
	\end{align}
	Combining \dref{decom}--\dref{de3}  proves
	\begin{align*}
	&n^{2}\big(
	\mathscr{C}_{\rm gcv}(K(\eta)) - \widehat{\sigma}^2(1+2p/n+3p^2/n^2)
	\big)/\sigma^4\\
	&=\frac{n^2}{\sigma^4}\Big(
	\Big(\frac{\|y - X\widehat{\beta} \|^2}{n}
	-\widehat{\sigma}^2 \Big)
	+\Big(\frac{\|y - X\widehat{\beta} \|^2}{n} \frac{2{\rm Tr}(M)}{n}
	-\widehat{\sigma}^2  \frac{2p}{n}\Big)\\
	&\hspace{8mm}+
	\Big(\frac{\|y - X\widehat{\beta} \|^2}{n} \frac{3  \big({\rm Tr}(M)\big)^2}{n^2}
	-\widehat{\sigma}^2  \frac{3p^2}{n^2}\Big)
	+\frac{\|y - X\widehat{\beta} \|^2}{n}
	O\Big(\frac1{n^3}\Big)\Big)\\
	&\xra{}W(\eta)
	\end{align*}
	in probability as $n\xra{}\infty$.
	
	Now, we will show $\widehat{\eta}_{\rm gcv} \xra{}\eta^*$.
	Under Assumption 2, there exists a compact set $\mathscr{X}_1\subset \Omega$ and positive constants $c_5,c_6$
	such that $\eta^*\subset \mathscr{X}_1$ and
	$0<c_5\leq \|K(\eta)\|\leq c_6<\infty$ for all
	$\eta\in \mathscr{X}_1$.
	Define the cost function
	\begin{align*}
	\overline{\mathscr{C}}_{\rm gcv}(K)\eq
	n^2\big(
	\mathscr{C}_{\rm gcv}(K) -\widehat{\sigma}^2 (1+2p/n+3p^2/n^2)
	\big)/\sigma^4.
	\end{align*}
	Note that $\widehat{\sigma}^2 $, $p$, $\sigma^4$, and $n$ do not depend on $\eta$.
	So, one yields
	$
	\widehat{\eta}_{\rm gcv}=\argmin_{\eta \in \Omega}
	\overline{\mathscr{C}}_{\rm gcv}(K(\eta)).
	$
	By Lemma \ref{ct} in appendix, in order to show $\widehat{\eta}_{\rm gcv} \xra{}\eta^*$,
	it suffices to show $\overline{\mathscr{C}}_{\rm gcv}(K(\eta))\xra{}W(\eta)$ in probability and uniformly in $\mathscr{X}_1$.
	Actually, the convergence of \dref{de1}, \dref{de2},  and \dref{de3} is also uniform in $\mathscr{X}_1$ by noting that $0<c_5\leq \|K(\eta)\|\leq c_6<\infty$,
	$\|K(\eta)^{-1}\|\leq 1/c_5$, $\|nG^{-1}\|\leq \|n(X^TX)\|$,
	$\|G^{-1}X^TX\| \leq \|I_p\|$, and $\|\widehat{\beta}_{\rm LS}\|=O_p(1)$ for all
	$\eta\in \mathscr{X}_1$.
	For example, consider the decomposition of \dref{de1}
	\begingroup
	\allowdisplaybreaks
	\begin{align}
	\label{de4}
	&\hspace{-15mm}
	n \widehat{\beta}_{\rm LS}^T
	K^{-1}G^{-1}X^TXG^{-1}K^{-1}
	\widehat{\beta}_{\rm LS}
	-\beta K^{-1}\Sigma^{-1}K^{-1}\beta \\
	\nonumber
	&=n\big(\widehat{\beta}_{\rm LS}-\beta \big)^T
	K^{-1}G^{-1}X^TXG^{-1}K^{-1}
	\widehat{\beta}_{\rm LS}\\
	\nonumber
	&\hspace{5mm}+\beta ^TK^{-1}\big(G^{-1}X^TX - I_p\big) nG^{-1}K^{-1}\widehat{\beta}_{\rm LS}\\
	\nonumber
	&\hspace{5mm}+\beta ^TK^{-1}\big( nG^{-1} -\Sigma^{-1} \big)K^{-1}\widehat{\beta}_{\rm LS}\\
	\nonumber
	&\hspace{5mm}+\beta ^TK^{-1}\Sigma^{-1}K^{-1} \big( \widehat{\beta}_{\rm LS}- \beta \big)
	\end{align}
	\endgroup
	It is easy to see that each term of the above equation converges uniformly in $\mathscr{X}_1$ by noting \dref{cd1} and \dref{lim1}.
	The uniform convergence of \dref{de2}  and \dref{de3} can be verified in a similar way.  This completes the proof. 

	
	\subsection{Proof of Theorem \ref{thm2}}
	We first consider the proof of \dref{fdloocv}. First, noting from the definition of $M$ in \eqref{hat} and Assumptions \ref{assum1} and \ref{ass3}.2, we have 
	$
	M_{ii}=O(1/ n),~i=1,\cdots,n,
	$
	which implies $1/(1-M_{ii})=1+O(1/ n)$, $i=1,\cdots,n$. 
	Then by using Lemma \ref{lem1}, Lemma \ref{lem4}, \dref{loocvfd} in Lemma \ref{lem3} and Lemma \ref{lem2}, we have
	\begingroup
	\allowdisplaybreaks
	\begin{align*}
	n^{2}\frac{\partial  \mathscr{C}_{\rm loocv}(K) }{\partial K}
	&=2\sigma^2 n K^{-1} G^{-1} X^T
	\Big(\sum_{i=1}^{n} \big(y_i-\widehat{y}_i\big)^2
	J_{ii} +O\Big(\frac1n\Big) \Big)
	X G^{-1}K^{-1}\\
	\nonumber
	&\hspace{-5mm}-2\sigma^2 n
	K^{-1} G^{-1} X^T
	\Big(\sum_{i,l=1}^{n}\big(y_i-\widehat{y}_i\big)y_lJ_{il}
	+O\Big(\frac1n\Big)
	\Big)
	X G^{-1}K^{-1}
	\\
	\nonumber
	&\hspace{-9mm}=2\sigma^2 n K^{-1} G^{-1}
	X^T
	L
	X G^{-1}K^{-1}\\
	\nonumber
	&\hspace{-5mm}-2\sigma^2  n
	K^{-1} G^{-1} X^T
	(I_n - X G^{-1}X^T)yy^T
	X G^{-1}K^{-1}+O_p(1/\sqrt{n})\\
	\nonumber
	&\hspace{-9mm}=2\sigma^2 n K^{-1} G^{-1}
	X^T
	L
	X G^{-1}K^{-1}\\
	\nonumber
	&\hspace{-5mm}-2\sigma^4  n
	K^{-1}
	G^{-1} K^{-1}G^{-1}X^Ty
	y^T
	X G^{-1}K^{-1}+O_p(1/\sqrt{n})\\
	\nonumber
	&\hspace{-9mm}=2\sigma^2 n K^{-1} G^{-1}
	X^T
	L
	X G^{-1}K^{-1}\\
	\nonumber
	&\hspace{-5mm}-2\sigma^4  n
	K^{-1}
	G^{-1} K^{-1}\widehat{\beta}
	\widehat{\beta}^TK^{-1}+O_p(1/\sqrt{n})\\
	\nonumber
	&\hspace{-9mm}\xra{}
	2\sigma^4K^{-1} \Sigma^{-1} K^{-1}
	-2\sigma^4K^{-1} \Sigma^{-1}
	K^{-1}
	\beta \beta ^T
	K^{-1}
	\end{align*}
	\endgroup
	in probability,
	where $L=\diag([\varepsilon_1^2,\varepsilon_2^2,\cdots,\varepsilon_n^2])$, the matrix $J_{ij}$ is defined in Lemma \ref{lem_J}, and the following results are used:
	\begingroup
	\allowdisplaybreaks
	\begin{align*}
	&
	\sum_{i=1}^{n} \big(y_i-\widehat{y}_i\big)^2 J_{ii}
	=\sum_{i=1}^{n}  \big[x_i^T\big(\beta -\widehat{\beta}\big) + \varepsilon_i\big]
	^2J_{ii}\\
	&\hspace{27.8mm}=\sum_{i=1}^{n}  \big[\varepsilon_i^2 +O_p(1/\sqrt{n})\big]J_{ii}
	=L +O_p(1/\sqrt{n}) I_n\\
	\nonumber
	&
	\sum_{i,l=1}^{n}\big(y_i-\widehat{y}_i\big)y_lJ_{il}
	=
	\sum_{i,l=1}^{n}\big((I_n - X G^{-1}X^T)y\big)_i y_lJ_{il} \\
	&\hspace{27.8mm}=(I_n - X G^{-1}X^T)yy^T\\
	&G^{-1} X^T(I_n - X G^{-1}X^T)y
	=G^{-1} (G - X^TX) G^{-1}X^Ty\\
	&\hspace{27.8mm}
	=
	\sigma^2G^{-1} K^{-1}G^{-1}X^Ty\\
	&X^T
	L X/n \xra{}\sigma^2\Sigma \text{ in probability, as } n\rightarrow\infty. 
	\end{align*}
	\endgroup
	
	Next, we consider the a.o to the first order of $\widehat{\eta}_{\rm loocv}$. By Lemma \ref{fdt},  to show the convergence of
	$\widehat{\eta}_{\rm loocv}$,
	it suffices to show that
	$
	{\rm Tr}\left(
	n^{2}\frac{\partial \mathscr{C}_{\rm loocv}(K)}{\partial K}
	\frac{\partial K(\eta)}{\partial \eta_i}\right)
	$ 
	converges in probability to
	$
	2\sigma^4{\rm Tr}\Big(K^{-1} \Sigma^{-1}
	K^{-1}
	\big(K  -
	\beta \beta ^T\big) K^{-1}
	\frac{\partial K(\eta)}{\partial \eta_i}\Big)
	$ for $1\leq i\leq m$ 
	uniformly in some compact subset $\mathscr{X}_2$ containing $\psi^*$ inside $\Omega$. To this goal, note that under Assumptions \ref{ass2} and \ref{ass3}, there exist  positive constants $c_7,c_8$
	such that   for all
	$\eta\in \mathscr{X}_2$, 
	$0<c_7\leq \|K(\eta)\|\leq c_8<\infty$ and $\|\frac{\partial K(\eta)}{\partial \eta_i}\|$, $i=1,\cdots,m$, are bounded from above. Then to prove the above uniform convergence is reduced to prove the uniform convergence of
	$n^2\frac{\partial \mathscr{C}_{\rm loocv}(K)}{\partial K}$ to
	$2\sigma^4K^{-1} \Sigma^{-1}
	K^{-1}
	\big(K-
	\beta \beta ^T\big) K^{-1}$ in probability in $\mathscr{X}_2$. This can be done in a similar way as the proof of the uniform convergence of $\overline{\mathscr{C}}_{\rm gcv}(K(\eta))$ to $W(\eta)$ in probability in $\mathscr{X}_1$ sketched in the proof of Theorem \ref{thm1} in the previous section. To save the space, the corresponding proof is omitted. This completes the proof of Theorem \ref{thm2}.

	\begin{rem}\label{rem_proof}
		
		The arguments on the reduction to the  convergence used in the proof of Theorem \ref{thm2} also apply to the proof of Theorems \ref{thm3}-\ref{thm5}. That is, it is sufficient to prove whether or not $n^2\frac{\partial \mathscr{C}_{\rm lkocv}(K)}{\partial K}$,
		$n^2\frac{\partial \mathscr{C}_{\rm rfcv}(K)}{\partial K}$, and
		$n^{3/2}\frac{\partial \mathscr{C}_{\rm hocv}(K)}{\partial K}$  in  $\Omega$ converge in probability  to 
		$cK^{-1} \Sigma^{-1}
		K^{-1}
		\big(K  \!-
		\beta \beta ^T\big) K^{-1}$, where $c$ is a constant and depends on the specific hyper-parameter estimator. In the proof of Theorem \ref{thm3}-\ref{thm5}, the above arguments will be adopted to save the space. 
		
	\end{rem}

	\subsection{Proof of Theorem \ref{thm3}}
	As mentioned in Remark \ref{rem_proof}, we only prove the  convergence of $n^2\frac{\partial \mathscr{C}_{\rm lkocv}(K)}{\partial K}$ in probability over $\Omega$. That is, by noting from \dref{lkocv} and \dref{iden2}, we need to prove 
	\begingroup
	\allowdisplaybreaks
	\begin{align}\label{lkocv_d}
	&\hspace{5mm}
	n^2\frac{\partial  \mathscr{C}_{\rm lkocv}(K) }{\partial K}
	=
	n^2\frac{1}{\tbinom{n}{k}} \frac{ \partial {\rm APE}_s}{\partial K}
	=n^2\frac{1}{\tbinom{n-1}{k-1}n/k} \frac{ \partial {\rm APE}_s}{\partial K}
	=\frac{nk}{\tbinom{n-1}{k-1}} \frac{ \partial {\rm APE}_s}{\partial K}\\\nonumber
	&=2\sigma^2\frac {n} {\tbinom{n-1}{k-1}}
	\sum_{s\in \mathcal{S}}
	K^{-1} G^{-1}
	X_s^TZ_s^{-2}
	(y_s \!-\! X_s\widehat{\beta})
	(y_s \!-\! X_s\widehat{\beta})^T
	Z_s^{-1} X_sG^{-1}K^{-1}\\
	\nonumber
	&\hspace{5mm}
	-2\sigma^2\frac {n} {\tbinom{n-1}{k-1}}
	\sum_{s\in \mathcal{S}}
	K^{-1} G^{-1}
	X_s^T
	Z_s^{-2} (y_s \!-\! X_s\widehat{\beta})
	y^TX
	G^{-1}K^{-1}\\
	\nonumber
	&=2\sigma^2 K^{-1}nG^{-1}
	\Big(\frac{1}{n\tbinom{n-1}{k-1}}
	\sum_{s\in \mathcal{S}}
	X_s^TZ_s^{-2}
	(y_s \!-\! X_s\widehat{\beta})
	(y_s \!- \! X_s\widehat{\beta})^TZ_s^{-1}X_s\Big)
	nG^{-1}K^{-1}\nonumber\\\nonumber
	&\hspace{5mm}
	-2\sigma^2
	K^{-1}n G^{-1}
	\Big(\frac {1} {\tbinom{n-1}{k-1}}
	\sum_{s\in \mathcal{S}}
	X_s^TZ_s^{-2}
	(y_s - X_s\widehat{\beta})\Big)
	y^TX
	G^{-1}K^{-1}\\\nonumber
	&\xra{}\left\{
	\begin{array}{cc}
	2\sigma^4
	(K^{-1}\Sigma^{-1}K^{-1}
	-  K^{-1}\Sigma^{-1}K^{-1}\beta \beta ^TK^{-1}), & \lambda =0\\
	\frac{2\sigma^4}{(1 - \lambda)^2}
	\big(K^{-1}\Sigma^{-1}K^{-1}
	-  K^{-1}\Sigma^{-1}K^{-1}\beta \beta ^TK^{-1}\big), & 0<\lambda <1
	\end{array}\right.
	\end{align}
	\endgroup
	in probability as $n\xra{}\infty$.

	We first consider the case {\it 1)}, i.e., $\lambda=0$.
	First, note that Assumption  \ref{ass3} implies that
	each element of $X_s^TX_s$ is bounded from above by $O(k)$
	and moreover Assumption  \ref{assum1} implies
	$
	X_s^TX_sG_{s^c}^{-1} = o(1)
	$
	due to $k/n\xra{}0$ as $n\xra{}\infty$. Then noting \dref{d1} and \dref{d2} shows
	\begin{align}
	X_s^TZ_s^{-1}
	=(1 +  o(1))X_s^T,  ~X_s^TZ_s^{-2} \!=  (1  \!+  o(1))X_s^T.
	\label{l2}
	\end{align}
	
	By using \dref{l2} and \dref{id3} and noting 
	\begin{align}\label{combination1}
	\frac{1}{\tbinom{n-1}{k-1}}\sum_{s\in \mathcal{S}}X_s^T(y_s - X_s\widehat{\beta})=X^T(y-X \widehat{\beta} ),\end{align}
	the following term in \eqref{lkocv_d}  can be written as follows:
	\begingroup
	\allowdisplaybreaks
	\begin{align}
	\label{l3}
	\frac{1}{\tbinom{n-1}{k-1}}
	&
	\sum_{s\in \mathcal{S}}X_s^TZ_s^{-2}
	(y_s -X_s\widehat{\beta})=\frac{1}{\tbinom{n-1}{k-1}}\sum_{s\in \mathcal{S}}(1 + o_p(1))X_s^T(y_s - X_s\widehat{\beta})\\
	\nonumber
	&=\frac{1}{\tbinom{n-1}{k-1}}(1 + o_p(1))\sum_{s\in \mathcal{S}}X_s^T(y_s - X_s\widehat{\beta})\\
	\nonumber
	&=(1 + o_p(1))X^T(y-X \widehat{\beta} )
	\\
	\nonumber
	&
	\xra{}\sigma^2K^{-1}\beta
	\end{align}in probability as $n\xra{}\infty$.
	\endgroup

	By using \dref{l2}, the following term in  \eqref{lkocv_d} can be written as follows: 
	\begingroup
	\allowdisplaybreaks
	\begin{align}
	\label{l4}
	\frac{1}{n\tbinom{n-1}{k-1}}
	&\sum_{s\in \mathcal{S}}
	X_s^TZ_s^{-2}(y_s - X_s\widehat{\beta})(y_s - X_s\widehat{\beta})^TZ_s^{-1}X_s\\
	\nonumber
	&=(1 + o_p(1))\frac{1}{n\tbinom{n-1}{k-1}}
	\sum_{s\in \mathcal{S}} X_s^T(y_s - X_s\widehat{\beta})(y_s - X_s\widehat{\beta})^TX_s.
	\end{align}
	Then by making use of \dref{a1}, we have 
	\begin{enumerate}
		\item It can be verified that 
		\begin{align}\label{combination2}
		\frac{1}{n\tbinom{n-1}{k-1}}
		\sum_{s\in \mathcal{S}} X_s^T \varepsilon_s\varepsilon_s^T X_s
		= X^T L X,\end{align} 
		where $L=\diag(\varepsilon) A\!~ \diag(\varepsilon)$, the main diagonal elements of $A$ are all equal to 1, and the $(kl)$-th non-diagonal element of $A$ with $k\neq l$
		is
		\begin{align*}
		A_{kl} = \tbinom{n-2}{k-2}/ \tbinom{n-1}{k-1}=(k-1)/(n-1) =o(1).
		\end{align*}
		Since $\frac 1{n^2} \sum_{1\leq k \neq l \leq n}A_{kl}^2=o(1)$, 
		by Lemma \ref{lem2}, we have
		\begin{align}
		\label{tt1}
		\frac{1}{n\tbinom{n-1}{k-1}}
		\sum_{s\in \mathcal{S}} X_s^T \varepsilon_s\varepsilon_s^T X_s
		&= X^T L X/n
		\xra{}\sigma^2 \Sigma
		\end{align}
		in probability as $n\xra{}\infty$.

		\item Noting that $X_s^T\varepsilon_s=O_p(\sqrt{k})$, $X^T\varepsilon=O_p(\sqrt{n})$ and $G^{-1} X_s^TX_s=O_p(k/n)$ yields
		\begingroup
		\allowdisplaybreaks
		\begin{align}
		\label{tt2}
		\frac{1}{n\tbinom{n-1}{k-1}}
		\sum_{s\in \mathcal{S}}
		&X_s^T\varepsilon_s\varepsilon^TX G^{-1} X_s^TX_s\\
		\nonumber
		&=\frac{1}{n\tbinom{n-1}{k-1}}
		\sum_{s\in \mathcal{S}} O_p(\sqrt{k})O_p(\sqrt{n})O_p(k/n)\\
		\nonumber
		&=\frac{1}{n\tbinom{n-1}{k-1}}\tbinom{n}{k}O_p(k\sqrt{k/n})
		= O_p(\sqrt{k/n})\xra{}0
		\end{align}in probability as $n\xra{}\infty$.
		\endgroup
		
		\item Noting that $X_s^T\varepsilon_s=O_p(\sqrt{k})$ and $G^{-1} X_s^TX_s=O_p(k/n)$ yields
		
		\begingroup
		\allowdisplaybreaks
		\begin{align}
		\label{tt3}
		\frac{1}{n\tbinom{n-1}{k-1}}
		\sum_{s\in \mathcal{S}}
		&\sigma^2X_s^T\varepsilon_s\beta ^TK^{-1}G^{-1} X_s^TX_s\\
		\nonumber
		&=\frac{1}{n\tbinom{n-1}{k-1}}\sum_{s\in \mathcal{S}} O_p(\sqrt{k})O_p(k/n)\\
		\nonumber
		&=\frac{1}{n\tbinom{n-1}{k-1}}\tbinom{n}{k}O_p(k\sqrt{k}/n)
		=O_p(\sqrt{k}/n)\xra{}0
		\end{align}in probability as $n\xra{}\infty$.
		\endgroup
		
		\item Noting that $\beta  - \widehat{\beta}=O_p(1/\sqrt{n})$, $X_s^T\varepsilon_s=O_p(\sqrt{k})$ and $X_s^TX_s=O_p(k)$ yields
		\begingroup
		\allowdisplaybreaks
		\begin{align}
		\label{tt4}
		&\hspace{-3mm}\frac{1}{n\tbinom{n-1}{k-1}}
		\sum_{s\in \mathcal{S}}
		X_s^TX_s
		(\beta  - \widehat{\beta})(y_s - X_s\widehat{\beta})^TX_s\\
		\nonumber
		&=\frac{1}{n\tbinom{n-1}{k-1}}
		\!\sum_{s\in \mathcal{S}}
		\!X_s^TX_s
		(\beta  \!-\! \widehat{\beta})(\varepsilon_s^TX_s  \!+\! (\beta \!-\!\widehat{\beta})^TX_s^TX_s)\\
		\nonumber
		&=\frac{1}{n\tbinom{n-1}{k-1}}
		\tbinom{n}{k}
		O_p(k) O_p(1/\sqrt{n}) \max(O_p(\sqrt{k}),O_p(k/\sqrt{n}))\\
		\nonumber
		&= \max(O_p(\sqrt{k/n}),O_p(k/n))\xra{}0
		\end{align}
		\endgroup 
		in probability as $n\xra{}\infty$.

	\end{enumerate}
	
	Finally, using \dref{l3} and \dref{l4}-\eqref{tt4} as well as \dref{cd1} and \dref{lim1} in Lemma \ref{lem4}, we conclude that \eqref{lkocv_d} is indeed true.

	\endgroup

We then consider the case {\it 2)}, i.e., $0<\lambda<1$. The idea of the proof is analogous to the case  {\it 1)}. In this case, 
Assumptions \ref{assum1} and \ref{assum4} imply that
\begingroup
\allowdisplaybreaks
\begin{subequations}\label{ap2}
	\begin{align}
	&X_s^TX_s/n =\lambda\Sigma+o(1),~G^{-1}X_s^TX_s = \lambda I_p +o(1) \\
	&X_s^TX_s G_{s^c}^{-1}= \lambda/  (1 - \lambda)  I_p  +o(1).
	\end{align}	
\end{subequations}
\endgroup
Then \dref{d1} and \dref{d2} yield
\begingroup
\allowdisplaybreaks
\begin{subequations}
	\label{l6}
	\begin{align}
	&X_s^TZ_s^{-1}
	=\Big(\frac{1}{1 -\lambda}  +  o(1)\Big)X_s^T,\\
	&X_s^TZ_s^{-2} = \Big(\frac{1}{(1 -\lambda)^2}  +  o(1)\Big)X_s^T.\label{l5}
	\end{align}
\end{subequations}
\endgroup

By using \dref{l5}, \dref{id3} and noting \eqref{combination1}, the following term in \eqref{lkocv_d}  can be written as follows:
\begingroup
\allowdisplaybreaks
\begin{align}
\label{l9}
\frac{1}{\tbinom{n-1}{k-1}}
&\sum_{s\in \mathcal{S}}X_s^TZ_s^{-2}(y_s - X_s\widehat{\beta})\\
\nonumber
&=\Big(  \frac{1}{(1 - \lambda)^2}  +o(1) \Big)
\frac{1}{\tbinom{n-1}{k-1}}
\sum_{s\in \mathcal{S}}X_s^T  (y_s - X_s\widehat{\beta})\\
\nonumber
&=\Big(  \frac{1}{(1 - \lambda)^2}  +o(1) \Big)X^T(y-X\widehat{\beta})\\
\nonumber
&\xra{}\frac{\sigma^2}{(1 - \lambda)^2} K^{-1}\beta
\end{align}in probability as $n\xra{}\infty$.
\endgroup

By using \dref{l6}, the following term in  \eqref{lkocv_d} can be written as follows: 
\begingroup
\allowdisplaybreaks
\begin{align}
\label{l8}
&\hspace{-3mm}
\frac{1}{n\tbinom{n-1}{k-1}}
\sum_{s\in \mathcal{S}}
X_s^TZ_s^{-2}(y_s - X_s\widehat{\beta})(y_s - X_s\widehat{\beta})^TZ_s^{-1}X_s\\
\nonumber
&=\Big(\frac{1}{(1 -\lambda)^3}  +  o(1)\Big)
\frac{1}{n\tbinom{n-1}{k-1}}
\sum_{s\in \mathcal{S}} X_s^T(y_s - X_s\widehat{\beta})(y_s - X_s\widehat{\beta})^TX_s.
\end{align}
\endgroup
Then by making use of \dref{a1}, we have 
\begin{enumerate}
	\item By using \dref{ap2} and \eqref{combination2}, we have
	\begingroup
	\allowdisplaybreaks
	\begin{align}
	\label{b1}
	\frac{1}{n\tbinom{n-1}{k-1}}
	&\sum_{s\in \mathcal{S}}
	X_s^T\varepsilon_s\varepsilon_s^TX_s
	-  X_s^T\varepsilon_s\varepsilon^TX G^{-1} X_s^TX_s
	\\
	\nonumber
	&=\frac{1}{n\tbinom{n-1}{k-1}}
	\sum_{s\in \mathcal{S}}
	X_s^T\varepsilon_s\varepsilon_s^TX_s
	-(\lambda +o(1) )X_s^T\varepsilon_s\varepsilon^TX
	\\
	\nonumber
	&=\frac1n
	\!\Big(
	X^T \diag(\varepsilon) A \diag(\varepsilon) X
	- (\lambda  +o(1) )
	X^T\varepsilon\varepsilon^TX
	\Big)\\
	\nonumber
	&=\frac1n
	\Big(
	X^T \diag(\varepsilon)
	\big(A -\lambda U \big) \diag(\varepsilon)X
	\Big)+o_p(1)\\
	\nonumber
	&=(1-\lambda)\frac1n
	X^T \diag(\varepsilon) \widetilde{A}~\! \diag(\varepsilon) X +o_p(1)
	\end{align} where 
	all of the main diagonal elements of $A$ are equal to 1 and the $(kl)$-th non-diagonal element of $A$ 
	is
	$A_{kl} =(k-1)/(n-1)
	$,
	$U$ denotes the matrix with all elements equal to 1, 
	and thus all of the main diagonal elements of $\widetilde{A}$ are one and its non-diagonal elements are $\big((k-1)/(n-1)-\lambda\big)/(1-\lambda)= o(1)$. Then by using Lemma \ref{lem2}, we have
	\begin{align}
	\frac{1}{n\tbinom{n-1}{k-1}}
	&\sum_{s\in \mathcal{S}}
	X_s^T\varepsilon_s\varepsilon_s^TX_s
	-  X_s^T\varepsilon_s\varepsilon^TX G^{-1} X_s^TX_s
	\xra{}(1-\lambda)\sigma^2 \Sigma\nonumber
	\end{align}
	in probability as $n\xra{}\infty$.
	
	\endgroup
	
	\item By using \dref{ap2} and $X^T\varepsilon=O_p(\sqrt{n})$,
	we have \begingroup
	\allowdisplaybreaks
	\begin{align}
	\label{b2}
	\frac{1}{n\tbinom{n-1}{k-1}}
	&\sum_{s\in \mathcal{S}}
	\sigma^2X_s^T\varepsilon_s\beta ^TK^{-1}G^{-1} X_s^TX_s\\
	\nonumber
	&=(\lambda  +o(1) )\frac{1}{n\tbinom{n-1}{k-1}}
	\sum_{s\in \mathcal{S}}
	\sigma^2X_s^T\varepsilon_s\beta ^TK^{-1}
	\\
	\nonumber
	&=\sigma^2(\lambda  +o(1) ) X^T \varepsilon \beta ^TK^{-1}/n\\
	&=O_p(1/\sqrt{n})\xra{}0\nonumber
	\end{align}
	in probability as $n\xra{}\infty$.
	\endgroup
	
	\item By using \dref{ap2}, \dref{id3} and $\beta  - \widehat{\beta}=O_p(1/\sqrt{n})$, we have
	\begingroup
	\allowdisplaybreaks
	\begin{align}
	\label{b3}
	\frac{1}{n\tbinom{n-1}{k-1}}
	&\sum_{s\in \mathcal{S}}
	X_s^TX_s
	(\beta  - \widehat{\beta})(y_s - X_s\widehat{\beta})^TX_s\\
	\nonumber
	&=
	( \lambda \Sigma +o(1)  )
	(\beta  - \widehat{\beta})
	\frac{1}{\tbinom{n-1}{k-1}}
	\sum_{s\in \mathcal{S}}
	(y_s - X_s\widehat{\beta})^TX_s\\
	\nonumber
	&= \big(\lambda\Sigma+o(1)\big)
	(\beta  - \widehat{\beta})
	(y - X\widehat{\beta})^TX\\
	\nonumber
	&=\sigma^2 \big(\lambda\Sigma+o(1)\big)
	(\beta  - \widehat{\beta})O_p(1)\\
	\nonumber
	&   = O_p(1/\sqrt{n})\xra{}0
	\end{align}
	in probability as $n\xra{}\infty$.
	\endgroup

\end{enumerate}

Finally, using \dref{l9} and \dref{l8}-\dref{b3}
as well as \dref{cd1} and \dref{lim1}, we conclude that \eqref{lkocv_d} is indeed true. This completes the proof.

\subsection{Proof of Theorem \ref{thm4}}

Similarly, for proving Theorem \ref{thm4}, we just verify the following limits hold in probability as $n\xra{}\infty$
\begingroup
\allowdisplaybreaks
\begin{align}
\label{rinf}
&\hspace{5mm}
n^2\frac{\partial  \mathscr{C}_{\rm rfcv}(K) }{\partial K}
=
n^2\frac{1}{r}\sum_{s\in \mathscr{S}}
\frac{ \partial {\rm APE}_s}{\partial K}
=nk\sum_{s\in \mathscr{S}}\frac{ \partial {\rm APE}_s}{\partial K}\\
\nonumber &=2\sigma^2n
\sum_{s\in \mathscr{S}}
K^{-1} G^{-1}
X_s^TZ_s^{-2}
(y_s \!-\! X_s\widehat{\beta})
(y_s \!-\! X_s\widehat{\beta})^T
Z_s^{-1} X_sG^{-1}K^{-1}\\
\nonumber
&\hspace{5mm}
-2\sigma^2n
\sum_{s\in \mathscr{S}}
K^{-1} G^{-1}
X_s^T
Z_s^{-2} (y_s \!-\! X_s\widehat{\beta})
y^TX
G^{-1}K^{-1}\\
\nonumber
&=2\sigma^2 K^{-1}nG^{-1}
\Big(\frac{1}{n}
\sum_{s\in \mathscr{S}}
X_s^TZ_s^{-2}
(y_s \!-\! X_s\widehat{\beta})
(y_s \!- \! X_s\widehat{\beta})^TZ_s^{-1}X_s\Big)
nG^{-1}K^{-1}\\
\nonumber
&\hspace{5mm}
-2\sigma^2
K^{-1}n G^{-1}
\Big(
\sum_{s\in \mathscr{S}}
X_s^TZ_s^{-2}
(y_s - X_s\widehat{\beta})\Big)
y^TX
G^{-1}K^{-1}\\
&
\xra{}
\left\{
\begin{array}{ll}
2\sigma^4
(K^{-1}\Sigma^{-1}K^{-1}
-  K^{-1}\Sigma^{-1}K^{-1}\beta \beta ^TK^{-1}),&if~r\xra{}\infty~ as ~n\xra{}\infty,\\
\frac{2\sigma^4r^2}{(r - 1)^2}
(K^{-1}\Sigma^{-1}K^{-1}
-  K^{-1}\Sigma^{-1}K^{-1}\beta \beta ^TK^{-1})&\\
~~~~~
+\frac{2\sigma^2r^2}{(r - 1)^2}
K^{-1}\Sigma^{-1}S'_n \Sigma^{-1} K^{-1}
+o_p(1),& ~if~ r \mbox{ is fixed}.
\end{array}
\right.\nonumber
\end{align}
\endgroup
It follows from $G^{-1} X_s^TX_s=O(k/n)$,  $X_s^T\varepsilon_s=O_p(\sqrt{k})$ and $\beta ^TK^{-1}=O(1)$  that
\begingroup
\allowdisplaybreaks
\begin{align}
\label{c2}
\frac1n
\sum_{s\in \mathscr{S}}
&\sigma^2X_s^T\varepsilon_s\beta ^TK^{-1}G^{-1} X_s^TX_s
\\
\nonumber
&=O_p(n^{-1}r\sqrt{k}k/n)
=O_p(\sqrt{k}/n)\\
\nonumber
&=O_p\Big(\sqrt{\frac1 {n r}} \Big)\xra{}0
\end{align}
\endgroup
in probability as $n\xra{}\infty$
for both $r\xra{}\infty$  and a fixed $r$ as $n\xra{}\infty$.

{\it Case 1} : $r\xra{}\infty$ as $n\xra{}\infty$.

It follows from the expressions \dref{d1} and \dref{d2} that
\begin{subequations}
	\begin{align}
	&X_s^TZ_s^{-1}(y_s - X_s\widehat{\beta})
	=\left(  1  +o(1) \right)X_s^T  (y_s - X_s\widehat{\beta})\\
	&X_s^TZ_s^{-2}(y_s - X_s\widehat{\beta})
	=\left(  1 +o(1) \right)X_s^T  (y_s - X_s\widehat{\beta}). \label{d3}
	\end{align}
\end{subequations}
Firstly, by \dref{d3} and \dref{id3}, the term involved in \dref{rinf} converges
\begin{align}
\label{c7}
\sum_{s\in \mathscr{S}}X_s^TZ_s^{-2}(y_s - X_s\widehat{\beta})
&=( 1  +o(1) )\sum_{s\in \mathscr{S}}X_s^T(y_s - X_s\widehat{\beta})\\
&=( 1  +o(1) )X^T(y-X^T\widehat{\beta})
\xra{}\sigma^2K^{-1}\beta  \nonumber
\end{align}
in probability as $n\xra{}\infty$.

Secondly, the term involved in \dref{rinf} has the approximation:
\begingroup
\allowdisplaybreaks
\begin{align}
\label{c6}
\frac1n
&\sum_{s\in \mathscr{S}} X_s^TZ_s^{-2}(y_s - X_s\widehat{\beta})(y_s - X_s\widehat{\beta})^TZ_s^{-1}X_s\\
\nonumber
&=( 1  +o(1) )
\frac1n
\sum_{s\in \mathscr{S}}X_s^T(y_s - X_s\widehat{\beta})(y_s - X_s\widehat{\beta})^TX_s.
\end{align}
\endgroup
Then by making use of  the decomposition \dref{a1}, we have 
\begin{enumerate}[1.]
	\item By using the estimate $G^{-1} X_s^TX_s = O(k/n)=O(1/r)=o(1)$ and  Lemma \ref{lem2}, one has
	\begingroup
	\allowdisplaybreaks
	\begin{align}
	\label{c5}
	\frac1n
	&\sum_{s\in \mathscr{S}}
	\Big(X_s^T\varepsilon_s\varepsilon_s^TX_s
	-  X_s^T\varepsilon_s\varepsilon^TX G^{-1} X_s^TX_s\Big)
	\\
	\nonumber
	&=\frac1n
	\sum_{s\in \mathscr{S}}
	\Big(X_s^T\varepsilon_s\varepsilon_s^TX_s
	-o(1) X_s^T\varepsilon_s\varepsilon^TX \Big)
	\\
	\nonumber
	&=\frac1n
	\!\Big(
	X^T L  X
	- o(1)
	X^T\varepsilon\varepsilon^TX
	\Big)\\
	\nonumber
	&=\frac1n
	X^T \diag(\varepsilon)
	A~ \diag(\varepsilon)X
	+o(1)O_p(1)\\
	\nonumber
	&\xra{}
	\sigma^2 \Sigma
	\end{align}
	\endgroup
	in probability as $n\xra{}\infty$
	since $\frac 1{n^2} \sum_{1\leq k \neq l \leq n}A_{kl}^2=\frac 1{n^2} (r(k^2-k))=(k-1)/n=1/r-1/n=o(1)$,
	where  $L=\diag([\varepsilon_{s_1}\varepsilon_{s_1}^T,\varepsilon_{s_2}\varepsilon_{s_2}^T,\cdots,\varepsilon_{s_r}\varepsilon_{s_r}^T])$ is a block diagonal matrix, $A$ is the matrix whose $(jl)$-th element is one when both $j$ and $l$ simultaneously belong to some $s_i$,  zero otherwise.

	\item Using the estimates $\beta  - \widehat{\beta}=O_p(1/\sqrt{n}) $,
	$X_s^T\varepsilon_s=O_p(\sqrt{k})$, and $X_s^TX_s=O(k)$ infers
	\begingroup
	\allowdisplaybreaks
	\begin{align}
	\label{c3}
	\frac1n
	&\sum_{s\in \mathscr{S}}
	X_s^TX_s
	(\beta  - \widehat{\beta})(y_s - X_s\widehat{\beta})^TX_s\\
	\nonumber
	&
	=\frac1n
	\sum_{s\in \mathscr{S}}
	X_s^TX_s
	(\beta  - \widehat{\beta})(\varepsilon_s^TX_s  + (\beta -\widehat{\beta})^TX_s^TX_s)\\
	\nonumber
	\nonumber
	&=O_p(n^{-1})rO_p(k) O_p(1/\sqrt{n}) \max(O_p(\sqrt{k}),O_p(k/\sqrt{n}))\\
	&=\max(O_p(\sqrt{1/r}),O_p(1/r))\xra{}0. \nonumber
	\end{align}
	\endgroup
\end{enumerate}

Finally, using \dref{c7} and  \dref{c2}, \dref{c6},  \dref{c5}, and \dref{c3} as well as \dref{cd1} and \dref{lim1} in Lemma \ref{lem4}, we conclude that \eqref{rinf} is indeed true.

{\it Case 2} : $r$ is a fixed positive integer as $n\xra{}\infty$.

It follows from the expressions \dref{d1} and \dref{d2} that
\begin{subequations}
	\label{c13}
	\begin{align}
	&X_s^TZ_s^{-1}(y_s - X_s\widehat{\beta})
	=\left(  \frac{r}{r- 1}  +o(1) \right)X_s^T  (y_s - X_s\widehat{\beta})\label{c9}\\
	&X_s^TZ_s^{-2}(y_s - X_s\widehat{\beta})
	=\left(  \frac{r^2}{(r - 1)^2}  +o(1) \right)X_s^T  (y_s - X_s\widehat{\beta}).\label{c10}
	\end{align}
\end{subequations}
Firstly, by \dref{c10} and \dref{id3}, the term involved in \dref{rinf} converges
\begin{align}
\label{c11}
\sum_{s\in \mathscr{S}}X_s^TZ_s^{-2}(y_s - X_s\widehat{\beta})
&=\Big(  \frac{r^2}{(r - 1)^2}  +o(1) \Big)
\sum_{s\in \mathscr{S}}X_s^T(y_s - X_s\widehat{\beta})\\
\nonumber
&=\Big(  \frac{r^2}{(r - 1)^2}  +o(1) \Big)
X^T(y - X\widehat{\beta})\\
&\xra{}\frac{r^2}{(r - 1)^2} \sigma^2K^{-1}\beta \nonumber
\end{align}
in probability as $n\xra{}\infty$.

Secondly, by \dref{c10} and \dref{id3}, the term involved in \dref{rinf} has the approximation:
\begingroup
\allowdisplaybreaks
\begin{align}
\label{c12}
\frac1n
&\sum_{s\in \mathscr{S}} X_s^TZ_s^{-2}(y_s - X_s\widehat{\beta})(y_s - X_s\widehat{\beta})^TZ_s^{-1}X_s
\\
\nonumber
&=\left(  \frac{r^3}{(r - 1)^3}  +o(1) \right)
\frac1n
\Big(\sum_{s\in \mathscr{S}}X_s^T(y_s - X_s\widehat{\beta})(y_s - X_s\widehat{\beta})^TX_s\Big).\end{align}
\endgroup
Then by making use of  the decomposition \dref{a1}, we have 
\begin{enumerate}[1.]
	\item By using the estimate $G^{-1} X_s^TX_s = 1/rI_p+ o(1)$  and  Lemma \ref{lem2}, under Assumptions \ref{ass3} and \ref{assum4},
	there holds
	\begingroup
	\allowdisplaybreaks
	\begin{align}
	\label{c4}
	&\frac1n\Big(\sum_{s\in \mathscr{S}}
	X_s^T\varepsilon_s\varepsilon_s^TX_s
	-  X_s^T\varepsilon_s\varepsilon^TX G^{-1} X_s^TX_s\Big)
	\\
	\nonumber
	&=\frac1n
	\Big(\sum_{s\in \mathscr{S}}
	X_s^T\varepsilon_s\varepsilon_s^TX_s
	-(r^{-1} +o(1) )X_s^T\varepsilon_s\varepsilon^TX \Big)
	\\
	\nonumber
	&=\frac1n
	\Big(
	X^T LX
	- (r^{-1} +o(1) )
	X^T\varepsilon\varepsilon^TX
	\Big)\\
	\nonumber
	\nonumber
	&=(1-r^{-1})\frac1n
	X^T
	(L''  + L')
	X +o_p(1) \\
	&=
	(1-r^{-1})(\sigma^2\Sigma +S'_n) +o_p(1) \nonumber
	\end{align}
	\endgroup
	where  $L=\diag([\varepsilon_{s_1}\varepsilon_{s_1}^T,\varepsilon_{s_2}\varepsilon_{s_2}^T,\cdots,\varepsilon_{s_r}\varepsilon_{s_r}^T])$ is a block diagonal matrix,
	$L''=\diag([\varepsilon_1\varepsilon_1,\cdots,\varepsilon_{n}\varepsilon_{n}^T])$, and $S'_n$ and $L'$ are defined as in \dref{li}.
	
	\item 
	By using Assumption \ref{assum4} that $X_s^TX_s/k\xra{}\Sigma$ as $n\xra{}\infty$, $\beta  - \widehat{\beta}=O_p(1/\sqrt{n})$, and \dref{id3},
	one has
	\begingroup
	\allowdisplaybreaks
	\begin{align}
	\label{c8}
	\frac1n
	&\Big(\sum_{s\in \mathscr{S}}
	X_s^TX_s
	(\beta  - \widehat{\beta})(y_s - X_s\widehat{\beta})^TX_s\Big)
	\\
	\nonumber
	&=
	(r^{-1}\Sigma +o(1)  )
	(\beta  - \widehat{\beta})
	\Big(\sum_{s\in \mathscr{S}}
	(y_s - X_s\widehat{\beta})^TX_s\Big)
	\\
	\nonumber
	&=(r^{-1}\Sigma +o(1)  )
	(\beta  - \widehat{\beta})
	(y - X\widehat{\beta})^TX
	\\
	&   = O_p\big(1/\sqrt{n}\big)O_p(1)
	=o_p(1). \nonumber
	\end{align}
	\endgroup
\end{enumerate}
Finally, using \dref{c11} and  \dref{c2}, \dref{c12}, \dref{c4}, and \dref{c8} as well as \dref{cd1} and \dref{lim1} in Lemma \ref{lem4}, we conclude that \eqref{rinf} is indeed true. The proof is completed.

\subsection{Proof of Theorem \ref{thm5}}

Similarly, for proving Theorem \ref{thm5}, we just verify the approximation: 
\begingroup
\allowdisplaybreaks
\begin{align} \label{holim}
&\hspace{5mm}
n^{3/2}\frac{\partial  \mathscr{C}_{\rm hocv}(K) }{\partial K}
=
n^{3/2}
\frac{ \partial {\rm APE}_{s}}{\partial K}\\
\nonumber&=2\sigma^2r\sqrt{n}
K^{-1} G^{-1}
X_s^TZ_s^{-2}
(y_s \!-\! X_s\widehat{\beta})
(y_s \!-\! X_s\widehat{\beta})^T
Z_s^{-1} X_sG^{-1}K^{-1}\\
\nonumber
&\hspace{5mm}
-2\sigma^2r\sqrt{n}
K^{-1} G^{-1}
X_s^T
Z_s^{-2} (y_s \!-\! X_s\widehat{\beta})
y^TX
G^{-1}K^{-1}\\
\nonumber
&=\frac1{\sqrt{n}}2\sigma^2 r K^{-1}nG^{-1}
\Big(\frac{1}{n}
X_s^TZ_s^{-2}
(y_s \!-\! X_s\widehat{\beta})
(y_s \!- \! X_s\widehat{\beta})^TZ_s^{-1}X_s\Big)
nG^{-1}K^{-1}\\
\nonumber
&\hspace{5mm}
-2\sigma^2r
K^{-1}n G^{-1}
\big(X_s^TZ_s^{-2}
(y_s - X_s\widehat{\beta})/\sqrt{n}\big)
y^TX
G^{-1}K^{-1}\\
\nonumber&=
\frac{2\sigma^2r^3}{(r - 1)^2}
K^{-1}\Sigma^{-1}
\left(\frac{1}{\sqrt{n}}
(X_s^T\varepsilon_s - X^T\varepsilon)\right)
\beta K^{-1}
+o_p(1)
+O_p\left(
\frac{1}{\sqrt{n}}\right).\nonumber
\end{align}
\endgroup

By \dref{d1} and \dref{d2} one has
\begin{subequations}
	\label{cc}
	\begin{align}
	&X_s^TZ_s^{-1}(y_s - X_s\widehat{\beta})
	=\left(  \frac{r}{r- 1}  +o(1) \right)X_s^T  (y_s - X_s\widehat{\beta})\\
	&X_s^TZ_s^{-2}(y_s - X_s\widehat{\beta})
	=\left(  \frac{r^2}{(r - 1)^2}  +o(1) \right)X_s^T  (y_s - X_s\widehat{\beta}).
	\end{align}
\end{subequations} 
Firstly, by using \dref{cc}  the term involved in \dref{holim} can be approximated as
\begin{align}
\label{d11}
X_s^TZ_s^{-2}(y_s &- X_s\widehat{\beta})
=\Big(  \frac{r^2}{(r - 1)^2}  +o(1) \Big)
X_s^T(y_s - X_s\widehat{\beta})\\
\nonumber
&=\frac{r^2}{(r - 1)^2}
\Big(X_s^T\varepsilon_s
- r^{-1}X^T\varepsilon\Big)
+o_p(\sqrt{n}), \nonumber
\end{align}
where the approximation 
\begin{align}
\label{a35}
&\hspace{5mm}
X_s^T(y_s - X_s\widehat{\beta})\\
\nonumber
&=X_s^T\varepsilon_s - X_s^TX_sG^{-1}X^T\varepsilon
+
\sigma^2X_s^TX_sG^{-1}K^{-1}  \beta\\
\nonumber
&=X_s^T\varepsilon_s
- (r^{-1}+o(1))X^T\varepsilon
+
\sigma^2(r^{-1}+o(1)) K^{-1}  \beta\\
\nonumber
&=\underbrace{X_s^T\varepsilon_s
	- r^{-1}X^T\varepsilon}_{O_p(\sqrt{n})}
+
\underbrace{\sigma^2r^{-1} K^{-1} \beta}_{O_p(1)}
+\underbrace{o(1))X^T\varepsilon}_{o_p(\sqrt{n})} + o(1)\\
&=X_s^T\varepsilon_s
- r^{-1}X^T\varepsilon
+o_p(\sqrt{n})\nonumber
\end{align}
is derived from \dref{id9}
and the scale of the term
\begin{align*}
X_s^T\varepsilon_s
- r^{-1}X^T\varepsilon
&=X_s^T\varepsilon_s
- r^{-1}\sum_{s'\in \mathscr{S}} X_{s'}^T\varepsilon_{s' }\\
&=r^{-1}\sum_{s'\in \mathscr{S}\backslash s} (X_s^T\varepsilon_s - X_{s'}^T\varepsilon_{s' })
=O_p(\sqrt{n}),
\end{align*}
which is induced by
\begin{align*}
&\hspace{-6mm}E\left(\frac{1}{\sqrt{n}}\sum_{s'\in \mathscr{S}\backslash s}(X_s^T\varepsilon_s - X_{s'}^T\varepsilon_{s' })
\frac{1}{\sqrt{n}}\sum_{s''\in \mathscr{S}\backslash s}
(X_s^T\varepsilon_s - X_{s''}^T\varepsilon_{s'' })^T\right)\\
&=\frac1n\sum_{s',s''} \sigma^2X_s^TX_s 
+\frac1n\sum_{s'}\sigma^2 X_{s'}^TX_{s'}
\\
&\xra{}
(r-1)\sigma^2\Sigma
\end{align*}
in terms of Assumption \ref{assum4} and \cite[Theorem 14.4-1, page 476]{Bishop2007}.

Secondly, by using \dref{cc} the term involved in \dref{holim} can be approximated as
\begingroup
\allowdisplaybreaks
\begin{align}
\label{d12}
\frac1n
X_s^TZ_s^{-2}&(y_s - X_s\widehat{\beta})(y_s - X_s\widehat{\beta})^TZ_s^{-1}X_s
\\
\nonumber
&=\left(  \frac{r^3}{(r - 1)^3}  +o(1) \right)
\frac1n
\Big(X_s^T(y_s - X_s\widehat{\beta})(y_s - X_s\widehat{\beta})^TX_s\Big).
\end{align}
\endgroup

By using the similar procedure in the proof of Lemma \ref{lem2} and \dref{a35},
one derives that
\begingroup
\allowdisplaybreaks
\begin{align}
\label{d4}
&\hspace{5mm}
\frac1n
\Big(X_s^T(y_s - X_s\widehat{\beta})(y_s - X_s\widehat{\beta})^TX_s\Big)
\\
\nonumber
&=\frac1n
\Big(
\big(X_s^T\varepsilon_s
- r^{-1}X^T\varepsilon
+o_p(\sqrt{n})\big)
\big(X_s^T\varepsilon_s
- r^{-1}X^T\varepsilon
+o_p(\sqrt{n})\big)^T
\Big)\\
\nonumber
&=\frac1n
\Big(
X_s^T\varepsilon_s \varepsilon_s  ^T X_s
+r^{-2} X^T\varepsilon \varepsilon^TX
-r^{-1} X_s^T\varepsilon_s \varepsilon^TX
-r^{-1} X^T\varepsilon \varepsilon_s^T X_s
+o_p(n)
\Big)\\
\nonumber
&=\frac1n
X(L_1 + L_2 + L_3 + L_4)X^T +o_p(1)=
O_p(1),
\end{align}
\endgroup
where  $L_1=\diag([0,\cdots,\varepsilon_{s}\varepsilon_{s}^T,\cdots,0])$ is a block diagonal matrix,
$L_2=r^{-2}U$ with the matrix  $U$ denoting all elements being one,  $L_3=-[0,\cdots,\varepsilon_{s}^T,\cdots,0]^T\varepsilon^T$ and $L_4 = L_3^T$.

Finally, by \dref{d11} and \dref{d12}--\dref{d4} as well as \dref{cd1} and \dref{lim1} in Lemma \ref{lem4}, we conclude that \eqref{holim} is indeed true. This completes the proof.

\subsection{Proof of Theorem \ref{thm6}}

The proof is straightforward by noting that the asymptotically optimal hyperparameters for the ridge regression estimator  are the same single value $\sigma^2{\rm Tr}\big(\Sigma^{-1} \big)/\beta^T\Sigma^{-1}\beta$  in the senses of both Definitions \ref{defn1} and \ref{defn2}. This completes the proof.

\end{appendix}

\begin{appendix}
\renewcommand{\thesection}{B}
\section{Technique Lemmas}

\setcounter{equation}{0}
\renewcommand{\theequation}{B\arabic{equation}}
\setcounter{lem}{0}
\renewcommand{\thelem}{B\arabic{lem}}
\setcounter{rem}{0}

%
%
%

\setcounter{subsection}{0}
\subsection{Matrix Differentials and Related Identities}


\begin{lem} \cite{Petersen2012}
	\label{md}
	Assume that  $X$ is a square matrix.
	Then
	\begin{align}
	\frac{\partial X^{-1}_{kl}}{\partial X_{ij}}=-X^{-1}_{ki} X^{-1}_{jl} \label{md3}
	\end{align}
	where $X_{ij}$  and $X^{-1}_{ij}$ denote the $(ij)$-th element of the matrix $X$ and $X^{-1}$, respectively.
	
\end{lem}

\begin{lem}\label{lem_J} 
	Let $J_{ij}$ with $i,j\in\mathbb N$ denote a matrix with compatible dimensions such that its elements are zero everywhere except the $(ij)$-th element which is equal to 1 and let $B,C\in\mathbb R^{l\times k}$ be any matrices. 
	Then the following identities hold:	
	\begingroup
	\allowdisplaybreaks
	\begin{subequations}
		\begin{align}
		&\sum_{i=1}^l\sum_{j=1}^kB_{ij}J_{ij} = B,\label{id1}\\
		& \sum_{a,b=1}^k B_{ia}B^T_{bj}J_{ab}= \sum_{a,b=1}^k B_{ia}B_{jb}J_{ab}=B^TJ_{ij}B,		\label{J1}\\
		& \sum_{i,j=1}^l b_ic_{j}J_{ij}= bc^T, \quad b,c\in\mathbb R^{l\times 1}.\label{J2}
		\end{align}
	\end{subequations}
	\endgroup
\end{lem}
\begin{proof}
	The proof is straightforward and thus omitted. 
\end{proof}

\begin{lem}\label{lem1}
	For the quantities $y_s,\varepsilon_s,X_s,Z_s,\widehat{\beta}_{s^c},G_{s^c},G$ defined in \eqref{rsc} and \eqref{ZsG}, and $y,X,\widehat \beta$ defined in \eqref{lrm} and \eqref{rls0}, 
	the following identities hold:
	\begingroup
	\allowdisplaybreaks
	\begin{subequations}
		\begin{align}
		&\|y_s - X_s\widehat{\beta}_{s^c}\|^2
		=\|Z_s^{-1}(y_s - X_s\widehat{\beta})\|^2 \label{iden1}\\
		&Z_s^{-1}=I_k  + X_s  G_{s^c}^{-1} X_s^T\label{iden5}\\
		&X_s^TZ_s^{-1}
		=X_s^T + X_s^TX_sG_{s^c}^{-1}X_s^T \label{d1}\\
		&X_s^TZ_s^{-2}=X_s^T + 2X_s^TX_sG_{s^c}^{-1}X_s^T
		+X_s^TX_sG_{s^c}^{-1}X_s^TX_sG_{s^c}^{-1}X_s^T\label{d2}\\
		\label{a1}
		&X_s^T(y_s - X_s\widehat{\beta})(y_s - X_s\widehat{\beta})^TX_s
		=X_s^T\varepsilon_s\varepsilon_s^TX_s
		-  X_s^T\varepsilon_s\varepsilon^TX G^{-1} X_s^TX_s \\
		\nonumber
		&
		\hspace{8mm}+\sigma^2X_s^T\varepsilon_s\beta ^TK^{-1}G^{-1} X_s^TX_s
		+X_s^TX_s
		(\beta  - \widehat{\beta})(y_s - X_s\widehat{\beta})^TX_s\\
		&\label{id9}
		X_s^T(y_s - X_s\widehat{\beta})
		=X_s^T\varepsilon_s - X_s^TX_sG^{-1}X^T\varepsilon
		+
		\sigma^2X_s^TX_sG^{-1}K^{-1}
		\end{align}
	\end{subequations}
	\endgroup
	Moreover, let $J_{ij}$ be the matrix defined in Lemma \ref{lem_J}. Then we have the following identities:	
	\begingroup
	\allowdisplaybreaks
	\begin{subequations}
		\begin{align}
		&\frac{\partial G^{-1}_{ij}}{\partial K} = \sigma^2K^{-T} G^{-T}
		J_{ij} G^{-T}K^{-T}\label{id10}\\
		&\frac{\partial M_{ij}}{\partial K} = \sigma^2K^{-T}
		G^{-T} X^T
		J_{ij}X G^{-T}K^{-T}. \label{id11}
		\end{align}
	\end{subequations}
	\endgroup
\end{lem}


\begin{proof}
	
	For proving \dref{iden1}, it suffices to show that
	\begin{align*}
	Z_s(y_s - X_s\widehat{\beta}_{s^c}) = y_s - X_s\widehat{\beta}
	\end{align*}	
	which is verified by
	\begingroup
	\allowdisplaybreaks
	\begin{align*}
	Z_s(y_s - X_s\widehat{\beta}_{s^c})
	&= (I_k - X_sG^{-1}X_s^T)(y_s - X_s\widehat{\beta}_{s^c}) \\
	&=y_s - X_s\big[( I_k -  G^{-1}X_s^T X_s )\widehat{\beta}_{s^c} +G^{-1}X_s^Ty_s\big]\\
	&=y_s - X_s \big[G^{-1}( X_{s^c}^T X_{s^c} +\sigma^2K^{1})\widehat{\beta}_{s^c} +G^{-1}X_s^Ty_s\big]\\
	&=y_s - X_s \big[G^{-1} X_{s^c}^T y_{s^c} +G^{-1}X_s^Ty_s\big]\\
	&=y_s - X_s\widehat{\beta}.
	\end{align*}
	\endgroup
	The identity \dref{iden5} follows from
	\begin{align*}
	G_{s^c}^{-1} - G^{-1}
	=G_{s^c}^{-1} ( X^TX + \sigma^2K^{-1} -X_{s^c}^T X_{s^c} -\sigma^2K^{-1})G^{-1}
	=G_{s^c}^{-1} X_{s}^T X_{s} G^{-1}.
	\end{align*}
	
	The identities \dref{d1}--\dref{d2} can be verified by a straightforward calculation.
	
	The decomposition \dref{a1} is derived in the way
	\begingroup
	\allowdisplaybreaks
	\begin{align*}
	&\hspace{6mm}
	X_s^T(y_s - X_s\widehat{\beta})(y_s - X_s\widehat{\beta})^TX_s\\
	\nonumber
	&=X_s^T( \varepsilon_s + X_s
	(\beta  - \widehat{\beta}) )(y_s - X_s\widehat{\beta})^TX_s\\
	\nonumber
	&=X_s^T\varepsilon_s
	(y_s - X_s\widehat{\beta})^TX_s
	+X_s^TX_s
	(\beta  - \widehat{\beta})(y_s - X_s\widehat{\beta})^TX_s\\
	\nonumber
	&=X_s^T\varepsilon_s\varepsilon_s^TX_s
	+X_s^T\varepsilon_s(\beta  - \widehat{\beta})^TX_s^TX_s
	+X_s^TX_s
	(\beta  - \widehat{\beta})(y_s - X_s\widehat{\beta})^TX_s\\
	\nonumber
	&=X_s^T\varepsilon_s\varepsilon_s^TX_s
	-  X_s^T\varepsilon_s\varepsilon^TX G^{-1} X_s^TX_s
	+\sigma^2X_s^T\varepsilon_s\beta ^TK^{-1}G^{-1} X_s^TX_s\\
	\nonumber
	&\hspace{4.5mm}
	+X_s^TX_s
	(\beta  - \widehat{\beta})(y_s - X_s\widehat{\beta})^TX_s.
	\end{align*}
	\endgroup
	by using the identities
	\begin{align*}
	&y_s - X_s\widehat{\beta}  = \varepsilon_s + X_s
	(\beta  - \widehat{\beta}) \\
	&\beta  - \widehat{\beta}
	=\sigma^2G^{-1}K^{-1}\beta
	-G^{-1}X^T \varepsilon.
	\end{align*}
	The equation \dref{id9} can be proved similarly.
	
	Using \dref{md3} gives
	\begingroup
	\allowdisplaybreaks
	\begin{align*}
	\frac{\partial G^{-1}_{ij}}{\partial K_{st}}
	&=\sum_{a,b=1}^p    \frac{\partial G^{-1}_{ij}}{\partial G_{ab}}
	\frac{\partial G_{ab}}{\partial K_{st}}
	=-\sum_{a,b=1}^pG^{-1}_{ia}G^{-1}_{bj}
	\Big(\sigma^2  \frac{\partial K^{-1}_{ab}}{\partial K_{st}}  \Big)\\
	&=\sigma^2 \sum_{a,b=1}^pG^{-1}_{ia}G^{-1}_{bj}
	K^{-1}_{as}K^{-1}_{tb}  \\
	&=\sigma^2 (K^{-T}G^{-T})_{si} (G^{-T}K^{-T})_{jt}
	\end{align*}
	\endgroup
	which implies \dref{id10}.
	Similarly, using \eqref{md3} and \eqref{J1} gives
	\begingroup
	\allowdisplaybreaks
	\begin{align*}
	\frac{\partial M_{ij}}{\partial K}
	&=\sum_{a,b=1}^p  X_{ia}  \frac{\partial G^{-1}_{ab}}{\partial K}X^T_{bj}
	=\sum_{a,b=1}^p
	X_{ia}
	\big( \sigma^2K^{-T} G^{-T}
	J_{ab} G^{-T}K^{-T} \big)
	X^T_{bj}
	\\
	&=\sigma^2K^{-T} G^{-T}
	\Big(\sum_{a,b=1}^pX_{ia}
	J_{ab}
	X^T_{bj}\Big)
	G^{-T}K^{-T} \\
	&=\sigma^2K^{-T} G^{-T}
	X^T
	J_{ij}
	X
	G^{-T}K^{-T}.
	\end{align*}
	\endgroup
	This completes the proof.
\end{proof}

\begin{lem}
	\label{lem4}
	Assumption \ref{assum1} yields the following limits
	\begin{subequations}
		\begin{align}
		\label{cd1}
		& \widehat{\beta}_{\rm LS}=(X^TX)^{-1}X^Ty\xra{}	\beta,~\widehat{\beta}=G^{-1}X^Ty\xra{}	\beta \\
		&O_p(1) = X^T(y-X \widehat{\beta} )  = \sigma^2K^{-1}G^{-1}X^Ty
		\xra{}  \sigma^2K^{-1} \beta\label{id3}\\
		\label{nid1}
		&\|y - X\widehat{\beta} \|^2 /n
		=\widehat{\sigma}^2
		+\sigma^4\widehat{\beta}_{\rm LS}^T
		K^{-1}G^{-1}X^TXG^{-1}K^{-1}\widehat{\beta}_{\rm LS}/n\\
		&\nonumber
		\hspace{23.3mm}=\widehat{\sigma}^2
		+O_p(1/n^2)\xra{}\sigma^2
		\end{align}
	\end{subequations}
	in probability as $n\xra{}\infty$,
	where $\widehat{\sigma}^2$ is defined in \dref{sigma}, while
	\begin{align}
	&G^{-1} X^TX \xra{} I_p,~ n(X^TX)^{-1}\xra{}\Sigma^{-1},
	~nG^{-1}\xra{}\Sigma^{-1} \label{lim1}
	\end{align}
	as $n\xra{}\infty$.
\end{lem}	
\begin{proof}
	The limits \dref{cd1}, \dref{id3}, and \dref{lim1} are straightforward and hence their proof is skipped.
	
	The result \dref{nid1} is derived as follows
	\begin{align*}
	\|y - X\widehat{\beta} &\|^2 /n
	=y^T(I_n - X(X^TX)^{-1}X^T)(I_n - X(X^TX)^{-1}X^T)y/n\\
	&\hspace{5mm}+2\sigma^2y^T
	(I_n - X(X^TX)^{-1}X^T)
	X G^{-1} K^{-1} (X^TX)^{-1}X^Ty/n\\
	&\hspace{5mm}+\sigma^4
	y^TX(X^TX)^{-1}
	K^{-1}G^{-1}
	X^TX
	G^{-1}K^{-1}
	(X^TX)^{-1}
	X^Ty/n\\
	&=y^T(I_n - X(X^TX)^{-1}X^T)y/n\\
	&\hspace{5mm}+\sigma^4\widehat{\beta}_{\rm LS}^T
	K^{-1}G^{-1}X^TXG^{-1}K^{-1}\widehat{\beta}_{\rm LS}/n
	\end{align*}
	by using the identity
	\begin{align*}
	y - X\widehat{\beta}
	&=(I_n - XG^{-1}X^T)y\\
	&=(I_n - X(X^TX)^{-1}X^T +\sigma^2 X G^{-1} K^{-1} (X^TX)^{-1}X^T)y.
	\end{align*}
	The proof is completed.
\end{proof}

\begin{lem}
	\label{lem3}
	The first order partial derivatives of
	the cost functions LKOCV \dref{loocv} and PEs \dref{pecv} with respect to $K$
	are given as follows, respectively.
	\begingroup
	\allowdisplaybreaks
	\begin{subequations}
		\begin{align}
		\label{loocvfd}
		&\frac{\partial  \mathscr{C}_{\rm  loocv}(K) }{\partial K}
		= \frac{2\sigma^2}{n}
		K^{-1} G^{-1} X^T
		\Big(\sum_{i=1}^n
		\frac{(y_i-\widehat{y}_i)^2}{(1-M_{ii})^3}  J_{ii}\Big)
		X G^{-1}K^{-1}\\
		\nonumber
		&\hspace{22mm}-\frac{2\sigma^2}{n}
		K^{-1} G^{-1} X^T
		\Big(\sum_{i,l=1}^n
		\frac{(y_i-\widehat{y}_i) y_l}{(1-M_{ii})^2} J_{il}\Big)
		X G^{-1}K^{-1}\\
		&		\label{iden2}
		\frac{ \partial {\rm APE}_s}{\partial K}
		=\frac{2\sigma^2}{k}
		K^{-1} G^{-1}
		X_s^TZ_s^{-2}
		(y_s \!-\! X_s\widehat{\beta})
		(y_s \!-\! X_s\widehat{\beta})^T
		Z_s^{-1} X_sG^{-1}K^{-1}\!
		\\
		\nonumber
		&\hspace{22mm}
		-\frac{2\sigma^2}{k}
		K^{-1} G^{-1}
		X_s^T
		Z_s^{-2} (y_s \!-\! X_s\widehat{\beta})
		y^TX
		G^{-1}K^{-1},
		\end{align}
	\end{subequations} where $J_{ij}$ is the matrix defined in Lemma \ref{lem_J}. 
	\endgroup
\end{lem}
\begin{proof}	
	We first prove \dref{loocvfd}. By using \dref{id11}, we have
	\begingroup
	\begin{align*}
	\frac{\partial  \big(y_i-\widehat{y}_i\big)}{\partial K}
	&=\frac{\partial  {\big(y - M y\big)_i}}{\partial K}
	=-\sum_{l=1}^n
	\frac{\partial  {M_{il}}}{\partial K} y_l\\
	&=-\sigma^2\sum_{l=1}^n K^{-T} G^{-T} X^T
	J_{il}X G^{-T}K^{-T} y_l\\
	\frac{\partial  \big(1-M_{ii}\big)}{\partial K}
	&=-\frac{\partial  M_{ii}}{\partial K}
	=-\sigma^2K^{-T} G^{-T} X^T
	J_{ii}X G^{-T}K^{-T},
	\end{align*} and then we have
	\allowdisplaybreaks
	\begin{align*}
	\frac{\partial  \mathscr{C}_{\rm  loocv}(K) }{\partial K}
	&=
	\frac1n\sum_{i=1}^n
	\frac{2(y_i-\widehat{y}_i)}{(1-M_{ii})^2}
	\frac{\partial (y_i-\widehat{y}_i)}{\partial K}
	-\frac2n\sum_{i=1}^n
	\frac{(y_i-\widehat{y}_i)^2}{(1-M_{ii})^3}
	\frac{\partial  \big(1-M_{ii}\big)}{\partial K}\\
	&=
	\frac1n\sum_{i=1}^n
	\frac{2(y_i-\widehat{y}_i)}{(1-M_{ii})^2}
	\!\Big(-\sigma^2\sum_{l=1}^n K^{-T} G^{-T} X^T
	J_{il}X G^{-T}K^{-T} y_l \Big)\\
	&\hspace{5mm}	-\frac2n
	\sum_{i=1}^n
	\frac{(y_i-\widehat{y}_i)^2}{(1-M_{ii})^3}
	\big(-\sigma^2K^{-T} G^{-T} X^T
	J_{ii}X G^{-T}K^{-T}\big)\\
	&=
	-\frac{2\sigma^2}{n}
	K^{-T} G^{-T} X^T
	\Big(\sum_{i,l=1}^n
	\frac{(y_i-\widehat{y}_i) y_l}{(1-M_{ii})^2} J_{il}\Big)
	X G^{-T}K^{-T}\\
	&\hspace{5mm}
	+\frac{2\sigma^2}{n}
	K^{-T} G^{-T} X^T
	\Big(\sum_{i=1}^n
	\frac{(y_i-\widehat{y}_i)^2}{(1-M_{ii})^3}  J_{ii}\Big)
	X G^{-T}K^{-T}\\
	&=
	-\frac{2\sigma^2}{n}
	K^{-1} G^{-1} X^T
	\Big(\sum_{i,l=1}^n
	\frac{(y_i-\widehat{y}_i) y_l}{(1-M_{ii})^2} J_{il}\Big)
	X G^{-1}K^{-1}\\
	&\hspace{5mm}
	+\frac{2\sigma^2}{n}
	K^{-1} G^{-1} X^T
	\Big(\sum_{i=1}^n
	\frac{(y_i-\widehat{y}_i)^2}{(1-M_{ii})^3}  J_{ii}\Big)
	X G^{-1}K^{-1}
	\end{align*}
	
	\endgroup
	
	\allowdisplaybreaks
	Next, we prove \dref{iden2}. By using \dref{md3}, \dref{id10} and \dref{id1}, we have
	\begingroup
	\allowdisplaybreaks
	\begin{align*}
	\frac{\partial (y_s \!-\! X_s\widehat{\beta})_i}{\partial K}
	&=-\sum_{j,l=1}^p (X_s)_{ij} \frac{\partial G^{-1}_{jl}}{\partial K}(X^Ty)_{l}\\
	&=-\sigma^2\sum_{j,l=1}^p (X_s)_{ij}
	K^{-T} G^{-T}J_{jl} G^{-T}K^{-T}
	(X^Ty)_{l}\\
	\frac{\partial (Z_s^{-1})_{ij}}{\partial K}
	&=\sum_{l,q=1}^k
	\frac{\partial (Z_s^{-1})_{ij}}{\partial (Z_s)_{lq}}
	\frac{\partial (Z_s)_{lq}}{\partial K}\\
	&	=-\sum_{l,q=1}^k
	(Z_s^{-1})_{il}
	(Z_s^{-1})_{qj}
	\Big(- \sum_{a,b=1}^p ( X_s)_{la}
	\frac{\partial G^{-1}_{ab} }{\partial K}
	(X_s^T)_{bq}  \Big)\\
	&= \sum_{l,q=1}^k\sum_{a,b=1}^p
	(Z_s^{-1})_{il}
	(Z_s^{-1})_{qj}
	( X_s)_{la}
	(X_s^T)_{bq}
	\sigma^2K^{-T} G^{-T}
	J_{ab} G^{-T}K^{-T}\\
	&
	=
	\sum_{a,b=1}^p
	(Z_s^{-1} X_s)_{ia}
	(Z_s^{-T} X_s)_{jb}
	\sigma^2K^{-T} G^{-T}
	J_{ab} G^{-T}K^{-T}\\
	& =\sigma^2K^{-T} G^{-T}
	X_s^TZ_s^{-1} J_{ij}Z_s^{-T} X_s G^{-T}K^{-T}.
	\end{align*}
	\endgroup
	Then by using \dref{pecv} and \eqref{J2}, we have
	\begin{align*}
	&
	\frac{\partial \|Z_s^{-1}(y_s - X_s\widehat{\beta})\|^2}{\partial K}
	=2\sum_{i,j=1}^k\big((y_s - X_s\widehat{\beta})^T
	Z_s^{-T}  \big)_{i}
	\frac{\partial (Z_s^{-1})_{ij}}{\partial K}
	(y_s \!-\! X_s\widehat{\beta})_j\\
	&\hspace{4cm}	+2\sum_{i=1}^k\big((y_s - X_s\widehat{\beta})^T
	Z_s^{-T} Z_s^{-1}\big)_{i}
	\frac{\partial (y_s \!-\! X_s\widehat{\beta})_i}{\partial K}\\
	& =2\sum_{i,j=1}^k
	\big((y_s - X_s\widehat{\beta})^T
	Z_s^{-T}  \big)_{i}
	(y_s \!-\! X_s\widehat{\beta})_j
	\sigma^2K^{-T} G^{-T}
	X_s^TZ_s^{-1} J_{ij}Z_s^{-T} X_s G^{-T}K^{-T}\\
	&\hspace{5mm}-2\sum_{i=1}^k \sum_{j,l=1}^p
	\big((y_s - X_s\widehat{\beta})^T
	Z_s^{-T} Z_s^{-1}\big)_{i}
	(X_s)_{ij}
	(X^TY)_{l}
	\sigma^2K^{-T} G^{-T}J_{jl} G^{-T}K^{-T}\\
	& =
	2\sigma^2K^{-T} G^{-T}
	X_s^TZ_s^{-1}
	Z_s^{-1} (y_s - X_s\widehat{\beta})
	(y_s \!-\! X_s\widehat{\beta})^T
	Z_s^{-T} X_sG^{-T}K^{-T}\\
	&\hspace{5mm}-2\sum_{j,l=1}^p
	\big(X_s^T Z_s^{-1} Z_s^{-T} (y_s - X_s\widehat{\beta})\big)_{j}
	(X^TY)_{l}
	\sigma^2K^{-T} G^{-T}J_{jl} G^{-T}K^{-T}\\
	& =
	2\sigma^2K^{-T} G^{-T}
	X_s^T
	Z_s^{-2} (y_s - X_s\widehat{\beta})
	(y_s \!-\! X_s\widehat{\beta})^T
	Z_s^{-T} X_sG^{-T}K^{-T}\\
	&\hspace{5mm}-2
	\sigma^2K^{-T} G^{-T}
	X_s^T Z_s^{-1} Z_s^{-T} (y_s - X_s\widehat{\beta})
	y^TX
	G^{-T}K^{-T}\\
	& = 2\sigma^2K^{-1} G^{-1}
	X_s^TZ_s^{-2}
	(y_s - X_s\widehat{\beta})
	(y_s \!-\! X_s\widehat{\beta})^T
	Z_s^{-1} X_sG^{-1}K^{-1}\\
	&\hspace{5mm}-2\sigma^2
	K^{-1} G^{-1}
	X_s^T
	Z_s^{-2} (y_s \!-\! X_s\widehat{\beta})
	y^TX
	G^{-1}K^{-1}.
	\end{align*}
	\endgroup
	This completes the proof. $\hfill\square$
	
	
	\begin{lem} \label{lem2}
		Let $A$ be a matrix of size $n\times n$ with $A_{ii}=1,1\leq i \leq n$ and $|A_{ij}|\leq 1,1\leq i\neq j \leq n$, where $A_{ij}$ denotes the $(ij)$-th element of $A$.
		Define the matrix $L=\diag(\varepsilon) A\!~ \diag(\varepsilon)$ of dimensions $n\times n$, namely, its $(ij)$-th element is $A_{ij}\varepsilon_i\varepsilon_j$.
		Suppose that 1) and 2) of Assumption \ref{ass3} hold.
		Thus, we have the following results:
		\begin{enumerate}[1)]
			\item if 1)  of Assumption \ref{assum1} holds and $X^TX/n=O(1)$, then $X^T L X/n  =O_p(1)$.

			\item if 1) and 2) of Assumption \ref{assum1} hold and $\frac 1{n^2} \sum_{ k \neq l }A_{kl}^2=o(1)$, then
			$$X^T L X/n \xra{}\sigma^2\Sigma$$
			in probability as $n\xra{}\infty$.
		\end{enumerate}
		%
		
	\end{lem}
	\begin{proof}
		We first consider the proof for part 1). First, it follows that
		\begin{align*}
		E[X^T L X/n  ]
		=\sigma^2 X^T X/n=O(1).
		\end{align*}
		Since the $(ij)$-th element of $X^T L X/n $ is 
		\begin{align*}
		\nonumber
		\big(X^T L X/n \big)_{ij}
		&=\frac1n\sum_{k,l=1}^n X^T_{ik} A_{kl}\varepsilon_k\varepsilon_lX_{lj}\\
		\nonumber
		&=\frac1n\sum_{k=1}^n X^T_{ik} X_{kj}\varepsilon_k^2
		+ \frac1n\sum_{k\neq l} X^T_{ik} A_{kl}X_{lj}\varepsilon_k\varepsilon_l
		\end{align*}
		and  $\Var(\varepsilon_k^2)\leq c_3$ by Assumption \ref{ass3} and $\Var(\varepsilon_k\varepsilon_l)=\sigma^4$ for $k\neq l$,
		the variance of the $(ij)$-th element of $X^T L X/n $ is bounded from above as follows:
		\begingroup
		\allowdisplaybreaks
		\begin{align}
		\label{var}&
		\Var\big( \big(X^T L X/n \big)_{ij}\big)\\
		\nonumber
		&=\frac{1}{n^{2}}\sum_{k=1}^{n}
		\big(X^T_{ik} X_{kj}\big)^2 \Var \big( \varepsilon_k^2 \big)
		+\frac{1}{n^{2}}\sum_{k\neq l}
		\big( X^T_{ik} A_{kl}X_{lj}  \big)^2
		\Var \big( \varepsilon_k\varepsilon_l \big)\\
		&\leq \frac{c_3}{n^{2}}\sum_{k=1}^{n}
		\big(X^T_{ik} X_{kj}\big)^2
		+\frac{\sigma^4}{n^{2}}\sum_{k\neq l}
		\big( X^T_{ik} A_{kl}X_{lj}  \big)^2.\nonumber
		\end{align}
		\endgroup
		Now applying Assumption \ref{ass3} and the assumption that $|A_{ij}|\leq 1$ derives
		\begin{align*}
		&\frac{1}{n^{2}}
		\sum_{k=1}^{n}
		\big(X^T_{ik}  X_{kj}\big)^2
		=\frac{1}{n^{2}}\sum_{k=1}^{n}O(1)
		=O\Big(\frac1n\Big),\\
		&\frac{1}{n^{2}}\sum_{k\neq l}
		\big( X^T_{ik} A_{kl}X_{lj}  \big)^2
		\leq  \frac{c_4^4}{n^{2}}\sum_{k\neq l} A_{kl}^2<\infty,
		\end{align*}
		which implies  $\Var\big( (X^T L X/n )_{ij}\big), 1\leq i,j\leq p$ are uniformly bounded from above by a constant. Then we  recall the following result, e.g., \cite[Theorem 14.4-1, p. 476]{Bishop2007} that for any random sequence $S_n$ with $E(S_n)=\mu_n$ and $\Var(S_n)<\infty$, then $S_n-\mu_n=O_p(1)$. Applying this result to $X^T L X/n$ yields 
		$X^T L X/n -  \sigma^2 X^T X/n  =O_p(1)$.
		Finally, note  that $
		\sigma^2 X^T X/n
		=O(1)$, and thus we have $X^T L X/n  =O_p(1)$.
		
		Next, we consider the proof for part 2). Under Assumption \ref{ass3} and the assumption that $\frac 1{n^2} \sum_{ k \neq l }A_{kl}^2=o(1)$, it is easy to see that the second term of \dref{var} tends to zero as $n\xra{}\infty$ and thus 
		$\Var\big( \big(X^T L X/n \big)_{ij}\big)=o(1)$.
		And further, it follows from 2) of Assumption \ref{assum1} that
		\begin{align*}
		E[(X^T L X/n  - \sigma^2 \Sigma)_{ij}]^2
		=\Var\big( \big(X^T L X/n \big)_{ij}\big)
		+(\sigma^2 (X^T X/n - \Sigma )_{ij})^2
		\xra{}0
		\end{align*}
		for all $1\leq i,j\leq p$, which implies that
		$X^T L X/n  \xra{} \sigma^2 \Sigma$ in mean square and also in probability as $n\xra{}\infty$.
		%
	\end{proof}
	

	\subsection{Convergence Results for Extremum Estimators}
	
	\begin{lem}\cite[Theorem 4.1.1, page 106]{Amemiya1985}
		\label{ct}
		Assume that
		\begin{enumerate}[1)]
			\item $g(\eta)$ is a deterministic function that is continuous in a compact subset $\mathscr{X}$ of $\mathbb{R}^m$ and is  minimized at the set
			\begin{align*}
			\eta^*\!\eq\!\argmin_{\eta \in \mathscr{X}} g(\eta) \!=\! \big\{\eta|\eta \in
			\mathscr{X}, g(\eta)\!=\!\min_{\eta'\in \mathscr{X} }g(\eta')\big\}
			\end{align*}
			
			\item A sequence of functions $\{g_n(\eta)\}$ converges to $g(\eta)$ in probability and uniformly in $\mathscr{X}$ as $n\xra{}\infty$.
			
		\end{enumerate}
		Then $\widehat{\eta}_n = \arg\min_{\eta\in \mathscr{X}} g_n(\eta)$
		converges to $ \eta^*$ in probability  as $n\xra{}\infty$, namely,
		\begin{align}
		\widehat{\eta}_n\xra{}\eta^*\label{cfc}
		\end{align}
		where the convergence \dref{cfc} is interpreted as
		\begin{align*}
		\inf_{a\in\widehat{\eta}_n,b\in\eta^*}
		\|a-b\|\xra{}0,~~\mbox{as}~n\xra{}\infty.
		\end{align*}
	\end{lem}

	\begin{lem}
		\label{fdt}
		\cite[Theorem 4.1.2, page 110]{Amemiya1985}
		Assume that
		\begin{enumerate}[1)]
			\item $f(\psi)$ is a continuous deterministic function from a compact set $\mathscr{X}$  of  $\!~\mathbb{R}^m$ to $\mathbb{R}^m$  and the set of roots of $f(\psi)$ in $\mathscr{X}$ is denoted by
			\begin{align*}
			\psi^*=\{\psi\in \mathscr{X} |f(\psi)=0\}
			\end{align*}
			
			\item  A function sequence $\{f_n(\psi)\}$ defined from $\mathscr{X}$ to $\mathbb{R}^m$ converges to  $f(\psi)$ in probability and uniformly in $\mathscr{X}$ as $n\xra{}\infty$.
			
		\end{enumerate}	
		Let $\widehat{\psi}_n$ be the set of roots of $f_n(\psi)$.
		Thus, we have
		\begin{align}
		\widehat{\psi}_n\xra{}\psi^*\label{fdc}
		\end{align}
		in probability as $n\xra{}\infty$, 
		where the convergence \dref{fdc} is interpreted as
		\begin{align*}
		\inf_{a\in\widehat{\psi}_n, b\in\psi^*}
		\|a-b\|\xra{}0,~~\mbox{as}~n\xra{}\infty.
		\end{align*}
	\end{lem}		
	
\end{appendix}

\end{document}